\documentclass[11pt]{amsart}

\usepackage{amscd}
\usepackage{amsmath, amssymb}
\usepackage{amsfonts}
\newcommand{\de}{\partial}
\newcommand{\db}{\overline{\partial}}
\newcommand{\dbar}{\overline{\partial}}

\newcommand{\ddbar}{\sqrt{-1} \partial \overline{\partial}}

\newcommand{\ov}[1]{\overline{#1}}

\newcommand{\ti}[1]{\tilde{#1}}
\newcommand{\vp}{\varphi}

\newcommand{\ve}{\varepsilon}

\renewcommand{\leq}{\leqslant}
\renewcommand{\geq}{\geqslant}

\newcommand{\be}{\begin{equation}}
\newcommand{\ee}{\end{equation}}

\begin{document}
\newcounter{remark}
\newcounter{theor}
\setcounter{remark}{0}
\setcounter{theor}{1}
\newtheorem{claim}{Claim}
\newtheorem{theorem}{Theorem}[section]
\newtheorem{lemma}[theorem]{Lemma}
\newtheorem{corollary}[theorem]{Corollary}
\newtheorem{proposition}[theorem]{Proposition}
\newtheorem{question}{question}[section]
\newtheorem{defn}{Definition}[theor]
\numberwithin{equation}{section}

\title[Degenerate complex Monge-Amp\`{e}re equations]{$C^{1,1}$ regularity of degenerate complex Monge-Amp\`{e}re equations and some applications}

\author{Jianchun Chu}
\address{Institute of Mathematics, Academy of Mathematics and Systems Science, Chinese Academy of Sciences, Beijing 100190, P. R. China}
\email{chujianchun@gmail.com}


\begin{abstract}
In this paper, we prove a $C^{1,1}$ estimate for solutions of complex Monge-Amp\`{e}re equations on compact almost Hermitian manifolds. Using this $C^{1,1}$ estimate, we show existence of $C^{1,1}$ solutions to the degenerate Monge-Amp\`{e}re equations, the corresponding Dirichlet problems and the singular Monge-Amp\`{e}re equations. We also study the singularities of the pluricomplex Green's function. In addition, the proof of the above $C^{1,1}$ estimate is valid for a kind of complex Monge-Amp\`{e}re type equations. As a geometric application, we prove the $C^{1,1}$ regularity of geodesics in the space of Sasakian metrics.
\end{abstract}

\maketitle

\section{Introduction}
Let $(M,\omega,J)$ be a compact almost Hermitian manifold of real dimension $2n$. We use $g$ and $\nabla$ to denote the corresponding Riemannian metric and Levi-Civita connection. In this paper, we consider the following complex Monge-Amp\`{e}re equation
\begin{equation}\label{CMAE}
\begin{split}
(\omega+\ddbar\vp)^{n} = {} & f \omega^{n}, \\
\quad \omega+\ddbar\vp > 0, \quad {} & \sup_{M}\vp = 0,
\end{split}
\end{equation}
where $f$ is a positive smooth function on $M$. Here we use $\ddbar\vp$ to denote $\frac{1}{2}(dJd\vp)^{(1,1)}$, which agrees with the standard notation when $J$ is integrable (see Section 2 for more explanations).

The complex Monge-Amp\`{e}re equation plays a significant role in complex geometry. When $(M,\omega,J)$ is K\"{a}hler, Yau \cite{Yau78} solved Calabi's conjecture (see \cite{Calabi57}) by proving the existence of solutions to (\ref{CMAE}). This is known as the Calabi-Yau theorem, which states that one can prescribe the volume form of a K\"ahler metric within a given K\"ahler class. There are many corollaries and applications of this result.

It is very interesting to extend the Calabi-Yau theorem to non-K\"{a}hler settings. When $(M,\omega,J)$ is Hermitian, the complex Monge-Amp\`{e}re equation has been studied under some assumptions on $\omega$ (see \cite{Cherrier87,Hanani96,GL10,TW10a,ZZ11}). In \cite{TW10b}, Tosatti-Weinkove solved (\ref{CMAE}) for any Hermitian metric $\omega$.

Recently, Chu-Tosatti-Weinkove \cite{CTW16} solved (\ref{CMAE}) on compact almost Hermitian manifolds. Unlike K\"{a}hler and Hermitian cases, almost Hermitian case is much more complicated. It is hard to obtain the complex Hessian estimate by the analogous computation. Instead, they considered a quantity involving the largest eigenvalue $\lambda_{1}$ of the real Hessian $\nabla^{2}\vp$. Combining the maximum principle and a series of delicate calculations, the real Hessian estimate was obtained. Following the approach of \cite{CTW16}, Chu-Tosatti-Weinkove \cite{CTW17} established the existence of $C^{1,1}$ solutions to the homogeneous complex Monge-Amp\`{e}re equation and solved the open problem of $C^{1,1}$ regularity of geodesics in the space of K\"{a}hler metrics (see \cite{Chen00}). Further applications of these ideas can be found in \cite{Tosatti17,CZ17,CTW18,CHZ17}.

However, the $C^{1,1}$ estimate in \cite{CTW16} depends on $\sup_{M}f$, $\sup_{M}|\de(\log f)|_{g}$ and lower bound of $\nabla^{2}(\log f)$. Hence, except the homogeneous complex Monge-Amp\`{e}re equation, it is impossible to apply this $C^{1,1}$ estimate in the study of degenerate complex Monge-Amp\`{e}re equation. Motivated by this, we prove the following estimate, which improve the above $C^{1,1}$ estimate.

\begin{theorem}\label{Main Theorem}
Let $\vp$ be a smooth solution of (\ref{CMAE}). Then there exists a constant $C$ depending only on $(M,\omega,J)$, $\sup_{M}f$, $\sup_{M}|\de(f^{\frac{1}{n}})|_{g}$ and lower bound of $\nabla^{2}(f^{\frac{1}{n}})$ such that
\begin{equation*}
\sup_{M}|\vp|+\sup_{M}|\de\vp|_{g}+\sup_{M}|\nabla^{2}\vp|_{g} \leq C,
\end{equation*}
where $\nabla$ is the Levi-Civita connection of $g$.
\end{theorem}

In \cite{Blocki12}, B{\l}ocki proved the similar estimates when $(M,\omega,J)$ is a compact K\"{a}hler manifold with nonnegative bisectional curvature.

We point out that the above $C^{1,1}$ estimate does not depend on the upper bound of $\nabla^{2}(f^{\frac{1}{n}})$, which is very important for Theorem \ref{Application SMAE} and \ref{Application SMAE measure}. Actually, in the proofs of Theorem \ref{Application SMAE} and \ref{Application SMAE measure}, we use $(|s|_{h}^{2}+i^{-1})^{\frac{1}{2}}$ to approximate $|s|_{h}$. However, there is no uniform upper bound of $\nabla^{2}(|s|_{h}^{2}+i^{-1})^{\frac{1}{2}}$ for $i\geq1$.

In addition, if we replace $f^{\frac{1}{n}}$ by $f^{\frac{1}{m}}$ for $m>n$, then Theorem \ref{Main Theorem} can be proved by the similar argument of \cite{CTW16}. But for $f^{\frac{1}{n}}$, it is impossible to adapt the approach of \cite{CTW16}. Some new techniques and auxiliary functions are needed. Later, we will discuss the proof of Theorem \ref{Main Theorem} in details.

For the degenerate complex Monge-Amp\`{e}re equation, it is well known that the solution may be only of class $C^{1,1}$ and not higher. As an application of Theorem \ref{Main Theorem}, we prove the existence of $C^{1,1}$ solutions.
\begin{theorem}\label{Application DMAE}
Let $(M,\omega,J)$ be a compact almost Hermitian manifold of real dimension $2n$. Suppose that $f$ is a nonnegative function on $M$ such that
\begin{equation*}
\sup_{M}f \leq C, \,\, \sup_{M}|\de(f^{\frac{1}{n}})|_{g}\leq C, \,\, \nabla^{2}(f^{\frac{1}{n}}) \geq -Cg,
\end{equation*}
for a constant $C$. If $f\not\equiv 0$, then there exists a pair $(\vp,b)$ where $\vp\in C^{1,1}(M)$ and $b\in\mathbb{R}$, such that
\begin{equation}\label{DMAE}
\begin{split}
(\omega+\ddbar\vp)^{n} = {} & fe^{b} \omega^{n}, \\
\quad \omega+\ddbar\vp \geq 0, \quad {} & \sup_{M}\vp = 0.
\end{split}
\end{equation}
\end{theorem}

In Theorem \ref{Application DMAE}, the function $f$ may not be $C^{1,1}$. The conclusion still holds when $f$ can be approximated by a sequence of smooth positive functions $f_{i}$ in the sense of $C^{0}$, such that
\begin{equation*}
\sup_{M}f_{i} \leq C, \,\, \sup_{M}|\de(f_{i}^{\frac{1}{n}})|_{g}\leq C, \,\, \nabla^{2}(f_{i}^{\frac{1}{n}}) \geq -Cg,
\end{equation*}
for a constant $C$ which is independent of $i$.

As an application of Theorem \ref{Application DMAE}, we show existence of $C^{1,1}$ solutions to the singular Monge-Amp\`{e}re equations.

\begin{theorem}\label{Application SMAE}
Let $(M,\omega)$ be a compact $n$-dimensional K\"{a}hler manifold and $L$ be a line bundle with Hermitian metric $h$. Given a section $s$ of $L$, a function $F\in C^{2}(M)$ and $N\geq n$ such that
\begin{equation}\label{Application SMAE equation 4}
\int_{M}|s|_{h}^{N}e^{F}\omega^{n} = \int_{M}\omega^{n},
\end{equation}
then there exists a solution $\vp\in C^{1,1}(M)\cap\textrm{PSH}(M,\omega)$ of
\begin{equation*}
(\omega+\ddbar\vp)^{n} = |s|_{h}^{N}e^{F}\omega^{n},
\quad \sup_{M}\vp=0.
\end{equation*}
\end{theorem}

Using blow-up construction, Theorem \ref{Application SMAE} can be applied in the study of the singularities of the pluricomplex Green's function. More precisely, we prove

\begin{theorem}\label{Application SMAE measure}
Let $(M,\omega)$ be a compact $n$-dimensional K\"{a}hler manifold with $\textrm{Vol}(M,\omega)=\int_{M}\omega^{n}=1$. Assume that $F$ is a smooth function on $M$ such that $\int_{M}e^{F}\omega^{n}=1$. Let $\delta_{p}$ be the Dirac measure concentrated at $p$. Then for $\ve_{0}$ sufficiently small, there exists $\vp\in\textrm{PSH}(M,\omega)\cap C^{1,1}(M\setminus\{p\})$ such that $\vp=\ve_{0}\log|z|^{2}+C^{1,1}$ in local coordinates centered at $p$ and
\begin{equation*}
\frac{(\omega+\ddbar\vp)^{n}}{\omega^{n}} = (1-\ve_{0})e^{F}+\ve_{0}\delta_{p}.
\end{equation*}
\end{theorem}

In \cite{CG09}, Coman-Guedj showed that there are examples of K\"{a}hler manifolds for which $\ve_{0}$ cannot be taken equal to $1$. Using the estimates of \cite{Yau78} and blow up argument, Phong-Sturm \cite{PS14} proved that $\vp=\ve_{0}\log|z|^{2}+C^{1,\alpha}$ near $z=p$ for any $\alpha\in(0,1)$. Our result, which makes use of Theorem \ref{Application DMAE}, improves this regularity to $C^{1,1}$.

For background material and further references on singular Monge-Amp\`{e}re equation, and relation to singular K\"{a}hler-Einstein metrics, we refer the reader to \cite{Tsuji88,Kolodziej98,TZ06,GZ07,EGZ09,BEGZ10,EGZ11,ST12,BBGZ13}. For further information, we refer to the survey \cite{Rubinstein14} and the references therein.

In the above theorems, we assume that $\de M=\emptyset$. When $\de M\neq\emptyset$, the Dirichlet problem has been studied extensively. Caffarelli-Kohn-Nirenberg-Spruck \cite{CKNS85} established the classical solvability for strongly pseudoconvex domains in $\mathbb{C}^{n}$. Guan \cite{Guan98} generalized this result to general domains under the assumption of existence of a subsolution. For further references, we refer the reader to \cite{Chen00,Blocki12,Plis14,HL15}.

Actually, Theorem \ref{Main Theorem} can be applied in the Dirichlet problem for the degenerate Monge-Amp\`{e}re equation. When $\de M\neq\emptyset$, Theorem \ref{Main Theorem} can be regarded as the interior estimate. In this case, the $C^{1,1}$ estimate depends on not only $(M,\omega,J)$, $\sup_{M}f$, $\sup_{M}|\de(f^{\frac{1}{n}})|_{g}$ and lower bound of $\nabla^{2}(f^{\frac{1}{n}})$, but also $\sup_{\de M}|\vp|$, $\sup_{\de M}|\de\vp|_{g}$ and $\sup_{\de M}|\nabla^{2}\vp|_{g}$. Combining this $C^{1,1}$ estimate with the boundary estimates (see \cite[Theorem 3.2']{Blocki12} and \cite[Lemma 7.16]{Boucksom12}), we obtain

\begin{theorem}\label{Application Dirichlet Problem}
Let $(M,\omega,J)$ be a compact $n$-dimensional K\"{a}hler manifold with nonempty smooth boundary, which we assume is weakly pseudoconcave (or Levi-flat). Suppose that $f$ is a nonnegative function on $M$ such that
\begin{equation*}
\sup_{M}f \leq C, \,\, \sup_{M}|\de(f^{\frac{1}{n}})|_{g}\leq C, \,\, \nabla^{2}(f^{\frac{1}{n}}) \geq -Cg,
\end{equation*}
for a constant $C$. We consider the Dirichlet problem
\begin{equation}\label{Dirichlet Problem equation 1}
(\omega+\ddbar\vp)^{n} = f\omega^{n}, \text{ on } M,\quad \vp=\vp_{0},\text{ on }\de M,
\end{equation}
where $\vp_{0}$ is a smooth function on $\de M$. If this problem admits a smooth subsolution, then there exists a solution $\vp\in C^{1,1}(M)\cap\textrm{PSH}(M,\omega)$ of (\ref{Dirichlet Problem equation 1}).
\end{theorem}

Theorem \ref{Application Dirichlet Problem} generalizes Corollary 1.3 in \cite{CTW17}, where Chu-Tosatti-Weinkove showed existence of $C^{1,1}$ solution to the homogeneous complex Monge-Amp\`{e}re equation (i.e., $f\equiv0$).

The complex Monge-Amp\`{e}re equation also plays an important role in Sasakian geometry. For the reader's convenience, let us recall the definition and some basic properties of Sasakian manifold. For general properties of Sasakian manifold, we refer the reader to the book \cite{BG08}. A Sasakian manifold $(N,g_{N})$ is a $(2m+1)$-dimensional Riemannian manifold such that the cone manifold
\begin{equation*}
(C(N),g_{c}) := (N\times \mathbb{R}^{+},r^{2}g_{N}+dr^{2})
\end{equation*}
is K\"{a}hler. There exists a Killing field $\xi$ of unit length on $N$, which is called Reeb vector field. We define tensor field $\Phi$ and contact $1$ form $\eta$ by
\begin{equation*}
\Phi(X) = \nabla_{X}\xi,~ \eta(X) = g_{N}(X,\xi) ~\text{~for any $X\in TN$},
\end{equation*}
where $\nabla$ is the Levi-Civita connection of $g_{N}$. We write $\mathcal{D}=\textrm{Ker}\{\eta\}$. Then $\Phi|_{\mathcal{D}}$ is a complex structure on $\mathcal{D}$, and $(\mathcal{D},\Phi|_{\mathcal{D}},d\eta)$ gives a transverse K\"{a}hler structure with K\"{a}hler form $\frac{1}{2}d\eta$ and Riemannian metric $g_{N}^{T}$ defined by
\begin{equation*}
g_{N}^{T}(X,Y) = \frac{1}{2}d\eta(X,\Phi|_{\mathcal{D}}(Y)) ~\text{~for any $X,Y\in\mathcal{D}$}.
\end{equation*}
Let $\mathcal{D}^{\mathbb{C}}$ be the complexification of $\mathcal{\mathcal{D}}$. We have the following decomposition
\begin{equation}\label{Decomposition of D}
\mathcal{D} = \mathcal{D}^{(1,0)}\oplus\mathcal{D}^{(0,1)},
\end{equation}
where $\mathcal{D}^{(1,0)}$ and $\mathcal{D}^{(0,1)}$ are the $\sqrt{-1}$ and $-\sqrt{-1}$ eigenspaces of $\Phi|_{\mathcal{D}}$.

A $p$ form $\theta$ on $(N,g_{N})$ is called basic if
\begin{equation*}
i_{\xi}\theta = 0 \text{~and~} L_{\xi}\theta = 0,
\end{equation*}
where $i_{\xi}$ is the contraction with $\xi$ and $L_{\xi}$ is the Lie derivative with respect to $\xi$.
In particular, when $p=0$, we use $C_{B}^{\infty}(N)$ to denote the set of all smooth basic functions on $N$, i.e.,
\begin{equation*}
C_{B}^{\infty}(N) = \{ \vp\in C^{\infty}(N)~|~\xi(\vp)=0 \}.
\end{equation*}
Let $\wedge_{B}^{p}(N)$ be the bundle of basic $p$ form. By (\ref{Decomposition of D}), there is a natural decomposition of its complexification
\begin{equation*}
\wedge_{B}^{p}(N)\otimes\mathbb{C} = \bigoplus_{i+j=p}\wedge_{B}^{i,j}(N),
\end{equation*}
where $\wedge_{B}^{i,j}(N)$ denotes the bundle of basic $(i,j)$ form. Accordingly, we define the corresponding operators $\de_{B}$ and $\dbar_{B}$ by
\begin{equation*}
\de_{B}: \wedge_{B}^{i,j}(N)\rightarrow\wedge_{B}^{i+1,j}(N),~
\dbar_{B}: \wedge_{B}^{i,j}(N)\rightarrow\wedge_{B}^{i,j+1}(N),
\end{equation*}
and set
\begin{equation*}
d_{B} = d|_{\wedge_{B}^{p}},~ d_{B}^{c} = \frac{1}{2}\sqrt{-1}(\dbar_{B}-\de_{B}).
\end{equation*}
It is clear that
\begin{equation*}
d_{B} = \de_{B}+\dbar_{B},~ d_{B}d_{B}^{c} = \sqrt{-1}\de_{B}\dbar_{B},~ d_{B}^{2}=0,~ (d_{B}^{c})^{2}=0.
\end{equation*}

Let $(N,g_{N})$ be a compact $(2m+1)$-dimensional Sasakian manifold. We write $\mathcal{H}$ for the space of Sasakian metrics, which can be parameterized by the space (see \cite{GZ10})
\begin{equation*}
\{ \vp\in C_{B}^{\infty}(N)~|~\eta_{\vp}\wedge(d\eta_{\vp})^{n}\neq0 \},
\end{equation*}
where
\begin{equation*}
\eta_{\vp} = \eta+d_{B}^{c}\vp,~
d\eta_{\vp} = d\eta+\sqrt{-1}\de_{B}\dbar_{B}\vp.
\end{equation*}
In \cite{GZ10}, Guan-Zhang introduced a geodesic equation in $\mathcal{H}$. For each Sasakian potential $\vp\in\mathcal{H}$, the tangent space $T_{\vp}\mathcal{H}$ is $C_{B}^{\infty}(N)$ and $d\mu_{\vp}=\eta_{\vp}\wedge(d\eta_{\vp})^{n}$ defines a measure on $N$. On this infinite dimensional manifold $\mathcal{H}$, the Riemannian metric is defined by
\begin{equation*}
(\psi_{1},\psi_{2})_{\vp} = \int_{N}\psi_{1}\psi_{2}d\mu_{\vp} ~\text{~for any $\psi_{1},\psi_{2}\in T_{\vp}\mathcal{H}$}.
\end{equation*}
For any $\vp_{0},\vp_{1}\in\mathcal{H}$, let $\vp:[0,1]\rightarrow\mathcal{H}$ be a path connecting them. The corresponding geodesic equation is
\begin{equation}\label{Geodesic equation 1}
\left\{ \begin{array}{ll}
  \frac{\de^{2}\vp}{\de t^{2}}-\frac{1}{4}\left|d_{B}\frac{\de\vp}{\de t}\right|_{g_{\vp}}^{2} = 0, \\[4mm]
  \vp(\cdot,0) = \vp_{0},~ \vp(\cdot,1) = \vp_{1},
\end{array}\right.
\end{equation}
where $g_{\vp}$ is the Sasakian metric determined by $\vp$, i.e.,
\begin{equation*}
g_{\vp} = \frac{1}{2}d\eta_{\vp}\circ(\textrm{Id}\otimes\Phi_{\vp})+\eta_{\vp}\otimes\eta_{\vp},~
\Phi_{\vp} = \Phi-\xi\otimes(d_{B}^{c}\vp)\circ\Phi.
\end{equation*}
In \cite{GZ12}, Guan-Zhang reduced (\ref{Geodesic equation 1}) to the Dirichlet problem of complex Monge-Amp\`{e}re type equation on the K\"{a}hler cone $N\times[1,\frac{3}{2}]\subset C(N)$. More precisely, they define a function $\psi$ and a $(1,1)$ form $\Omega_{\psi}$ on $N\times[1,\frac{3}{2}]$ by
\begin{equation*}
\psi(\cdot,r) = \vp(\cdot,2r-2)+4\log r
\end{equation*}
and
\begin{equation}\label{Definition of Omega psi}
\Omega_{\psi} = \omega_{c}+\frac{r^{2}}{2}\sqrt{-1}\left(\de\dbar\psi-\frac{\de\psi}{\de r}\de\dbar r\right),
\end{equation}
where $\omega_{c}$ is the K\"{a}hler form of $(C(N),g_{c})$. As noted in \cite{GZ12}, the path $\vp$ is a geodesic connecting $\vp_{0}$ and $\vp_{1}$ if and only if $\psi$ solves the following Dirichlet problem on $N\times[1,\frac{3}{2}]$
\begin{equation}\label{Geodesic equation 2}
\left\{ \begin{array}{ll}
  (\Omega_{\psi})^{m+1} = 0, \\[2mm]
  \Omega_{\psi} \geq 0, \\[2mm]
  \psi(\cdot,1) = \psi_{1},~ \psi(\cdot,\frac{3}{2}) = \psi_{\frac{3}{2}}.
\end{array}\right.
\end{equation}
In order to solve (\ref{Geodesic equation 2}), for any $\ve\in(0,1)$, Guan-Zhang \cite{GZ12} considered the perturbation geodesic equation
\begin{equation}\label{Perturbation geodesic equation}
\left\{ \begin{array}{ll}
  (\Omega_{\psi})^{m+1} = \ve f\omega_{c}^{m+1}, \\[2mm]
  \Omega_{\psi} > 0, \\[2mm]
  \psi(\cdot,1) = \psi_{1},~ \psi(\cdot,\frac{3}{2}) = \psi_{\frac{3}{2}},
\end{array}\right.
\end{equation}
where $f$ is a positive basic function. They proved that there exists a smooth solution $\psi_{\ve}$ of (\ref{Perturbation geodesic equation}), and established the $C_{w}^{2}$ estimate (see \cite[Theorem 1]{GZ12})
\begin{equation*}
\|\psi_{\ve}\|_{C_{w}^{2}(N\times[1,\frac{3}{2}],g_{c})}
:= \|\psi_{\ve}\|_{C^{1}(N\times[1,\frac{3}{2}],g_{c})}+\sup_{N\times[1,\frac{3}{2}]}|\Delta_{c}\psi_{\ve}| \leq C,
\end{equation*}
where $\Delta_{c}$ is the Laplace-Beltrami operator of $g_{c}$ and $C$ is a constant depending only on $(N,g_{N})$, $\|f^{\frac{1}{m}}\|_{C^{2}(N\times[1,\frac{3}{2}],g_{c})}$, $\|\psi_{1}\|_{C^{2,1}(N,g_{N})}$ and $\|\psi_{\frac{3}{2}}\|_{C^{2,1}(N,g_{N})}$. Letting $\ve\rightarrow0$, Guan-Zhang showed existence of $C_{w}^{2}$ solution of (\ref{Geodesic equation 2}). This implies that any two Sasakian potentials $\vp_{0},\vp_{1}$ can be joined by a $C_{w}^{2}$ geodesic. Clearly, this geodesic is $C^{1,\alpha}$ for any $\alpha\in(0,1)$.

When $m=1$, (\ref{Geodesic equation 2}) is equivalent to the geodesic equation in the space of volume forms on Riemannian manifold with fixed volume (see \cite{Donaldson10}). In this setting, Chen-He \cite{CH11} proved the geodesic is $C_{w}^{2}$, and Chu \cite{Chu18} improved this regularity to $C^{1,1}$.

Actually, Sasakian geometry can be considered as odd dimensional counterpart of K\"{a}hler geometry. The space of K\"{a}hler metrics can be endowed with a natural Riemannian structure (see \cite{Mabuchi87,Semmes92,Donaldson99}). Chen \cite{Chen00} showed any two K\"{a}hler potentials can be connected by a $C_{w}^{2}$ geodesic. As mentioned before, Chu-Tosatti-Weinkove \cite{CTW17} improved this regularity to $C^{1,1}$.

In two settings mentioned above, the $C^{1,1}$ regularity is optimal (see \cite{LV13,DL12,Darvas14}). It was expected that analogous result can be proved in the Sasakian case. In this paper, we prove the $C^{1,1}$ regularity of geodesics in the space of Sasakian metrics.

\begin{theorem}\label{C11 regularity of geodesics}
Let $(N,g_{N})$ be a compact $(2m+1)$-dimensional Sasakian manifold. For any two Sasakian potentials $\vp_{0},\vp_{1}\in\mathcal{H}$, the geodesic connecting them is $C^{1,1}$.
\end{theorem}

To prove Theorem \ref{C11 regularity of geodesics}, it suffices to establish the $C^{1,1}$ estimate for the perturbation geodesic equation (\ref{Perturbation geodesic equation}). Since $\Omega_{\psi}$ involves the first order term (see (\ref{Definition of Omega psi})), (\ref{Perturbation geodesic equation}) is much more complicated than the standard complex Monge-Amp\`{e}re equation. Fortunately, the proof of Proposition \ref{Second order estimate} is still valid for (\ref{Perturbation geodesic equation}). In Section 6, we will introduce a kind of complex Monge-Amp\`{e}re type equation (\ref{Complex Monge-Ampere type equation}), and (\ref{Perturbation geodesic equation}) can be regarded as a special case of (\ref{Complex Monge-Ampere type equation}). By the same proof of Proposition \ref{Second order estimate}, we derive the $C^{1,1}$ interior estimate for (\ref{Complex Monge-Ampere type equation}) (see Proposition \ref{C11 interior estimate}), which gives an extension of Proposition \ref{Second order estimate}. Then Theorem \ref{C11 regularity of geodesics} follows from Proposition \ref{C11 interior estimate} and \cite[Theorem 1, Proposition 3]{GZ12} ($C_{w}^{2}$ estimate and $C^{1,1}$ boundary estimate).

\bigskip

We now discuss the proof of Theorem \ref{Main Theorem}. Zero order estimate was proved in \cite{CTW16}. For the first order estimate, we adapt an approach of B{\l}ocki \cite[Theorem 1]{Blocki09} in the K\"{a}hler case. However, there are more troublesome terms arising from the non-integrability. We show that these terms can be controlled in Section 3.

The heart of this paper is Section 4, where we prove the second order estimate. Compared to the second order estimate of \cite{CTW16}, our method is quite different. The main reason is that the concavity of $(\det\ti{g})^{\frac{1}{n}}$ is weaker than that of $\log\det\ti{g}$. Then there are less "good" third order terms when we differentiate the equation twice. Hence, it is impossible to control "bad" third order terms by the similar argument in \cite{CTW16}.

In order to overcome this difficulty, we apply the maximum principle to a new quantity. Compared to the quantity of \cite{CTW16}, we add a new term involving $|\ti{\omega}|_{g}^{2}$. Crucially, this gives more "good" third order terms, which can be used to control "bad" third order terms. On the other hand, we use covariant derivatives with respect to the Levi-Civita connection $\nabla$. Then there is no third order term when we commute derivatives (see (\ref{Commutation formula}) ($k=2$)). And this is the main reason why we do not use the Chern connection. For general almost Hermitian manifold, $(\nabla^{2}\vp)^{(1,1)}$ is different from $\de\dbar\vp$ (they coincide in the K\"{a}hler case). We introduce a new tensor field $S$ (see (\ref{Definition of S})) to describe this difference. Because of this, more "bad" third order terms appear when we differentiate the equation twice. Fortunately, these terms can be controlled by using the maximum principle (see (\ref{Main inequality equation 2}), (\ref{Main inequality equation 3})). After a series of delicate calculations and estimates, we prove the second order estimate.

We expect that the method we introduced in this paper will adapt to other nonlinear PDEs on compact almost Hermitian manifolds.

\section{Basic results and notation}
Let $M$ be a compact manifold of real dimension $2n$. Recall that an almost complex structure $J$ on $M$ is a bundle automorphism of the tangent bundle $TM$ satisfying $J^{2}=-\textrm{Id}$. Let $T^{\mathbb{C}}M$ be the complexified tangent space. Then we have the natural decomposition
\begin{equation*}
T^{\mathbb{C}}M = T_{\mathbb{C}}^{(1,0)}M\oplus T_{\mathbb{C}}^{(0,1)}M,
\end{equation*}
where $T_{\mathbb{C}}^{(1,0)}M$ and $T_{\mathbb{C}}^{(0,1)}M$ are the $\sqrt{-1}$ and $-\sqrt{-1}$ eigenspaces of $J$. For any $1$ form $\alpha$ on $M$, we define
\begin{equation*}
J\alpha(V) = -\alpha(JV) \quad \text{for any $V\in TM$.}
\end{equation*}
Then we have the similar decomposition of $T^{\mathbb{C}}M^{*}$ into the $\sqrt{-1}$ and $-\sqrt{-1}$ eigenspaces, spanned by the $(0,1)$ and $(1,0)$ forms respectively. And every $k$ form can be expressed uniquely as a linear combination of $(p,q)$ forms.

Let $g$ be a Riemannian metric on $M$. $(M,g,J)$ is called an almost Hermitian manifold if
\begin{equation*}
g(V_{1},V_{2}) = g(JV_{1},JV_{2}) \quad \text{for any $V_{1},V_{2}\in TM$.}
\end{equation*}
We define $(1,1)$ form $\omega$ by
\begin{equation*}
\omega(V_{1},V_{2}) =  g(JV_{1},V_{2}) \quad \text{for any $V_{1},V_{2}\in TM$.}
\end{equation*}
It then follows that
\begin{equation*}
g(V_{1},V_{2}) = \omega(V_{1},JV_{2}).
\end{equation*}
Hence, we often use $(M,\omega,J)$ to denote $(M,g,J)$ for convenience.

For any $(p,q)$ form $\beta$, we define
\begin{equation*}
\de\beta = (d\beta)^{p+1,q} \text{~and~} \dbar\beta = (d\beta)^{p,q+1}.
\end{equation*}
By direct calculation, for any $f\in C^{2}(M)$, we have
\begin{equation*}
\ddbar f = \frac{1}{2}(dJdf)^{(1,1)}.
\end{equation*}
For any two $(1,0)$ vector fields $X,Y$, we also have the following formula (see e.g. \cite[(2.5)]{HL15})
\begin{equation}\label{ddbar formula}
(\de\db \vp)(X,\ov{Y}) = X\ov{Y}(\vp)-[X,\ov{Y}]^{(0,1)}(\vp).
\end{equation}

Let $\{e_{i}\}_{i=1}^{n}$ be a local frame for $T_{\mathbb{C}}^{(1,0)}M$. Throughout this paper, we use covariant derivatives with respect to the Levi-Civita connection $\nabla$. And the subscripts of a function $f$ always denote the covariant derivatives of $f$ with respect to $\nabla$, e.g.,
\begin{equation*}
f_{i} = \nabla_{e_{i}}f, \,\, f_{i\ov{j}} = \nabla_{\ov{e}_{j}}\nabla_{e_{i}}f, \,\,
f_{i\ov{j}k} = \nabla_{e_{k}}\nabla_{\ov{e}_{j}}\nabla_{e_{i}}f.
\end{equation*}
Recalling the commutation formula for covariant derivatives (Ricci identity), for any two vector fields $V_{1},V_{2}$, we have
\begin{equation}\label{Commutation formula}
\nabla_{V_{1}}\nabla_{V_{2}}(\nabla^{k}f)-\nabla_{V_{2}}\nabla_{V_{1}}(\nabla^{k}f) = (\nabla^{k}f)*\textrm{Rm},
\end{equation}
where $\textrm{Rm}$ is the curvature tensor of $g$ and $*$ denotes a contraction.

Next we define a tensor field $S$ by
\begin{equation}\label{Definition of S}
\nabla_{e_{i}}\ov{e}_{j}-[e_{i},\ov{e}_{j}]^{(0,1)} = S_{i\ov{j}}^{p}e_{p}+S_{i\ov{j}}^{\ov{p}}\ov{e}_{p}.
\end{equation}
By direct calculation, it is clear that
\begin{equation}\label{Property of S}
\ov{S_{i\ov{j}}^{p}} = S_{j\ov{i}}^{\ov{p}} ~\text{~and~}~
\ov{S_{i\ov{j}}^{\ov{p}}} = S_{j\ov{i}}^{p}.
\end{equation}
For convenience, we write
\begin{equation*}
\ti{\omega} = \omega+\ddbar\vp > 0
\end{equation*}
and let $\ti{g}$ be the corresponding Riemannian metric. Combining (\ref{ddbar formula}) and (\ref{Definition of S}), we see that
\begin{equation}\label{Expression of tilde g}
\begin{split}
\ti{g}_{i\ov{j}} = {} & g_{i\ov{j}}+(\de\dbar\vp)(e_{i},\ov{e}_{j}) \\
                 = {} & g_{i\ov{j}}+\vp_{i\ov{j}}+S_{i\ov{j}}^{p}\vp_{p}+S_{i\ov{j}}^{\ov{p}}\vp_{\ov{p}},
\end{split}
\end{equation}
where $\vp_{i\ov{j}}=(\nabla^{2}\vp)(e_{i},\ov{e}_{j})$.

For later use, let us recall the $L^{1}$ estimate and zero order estimate.
\begin{proposition}[Proposition 2.3 of \cite{CTW16}]\label{L1 estimate}
For any $\vp\in C^{\infty}(M)$ satisfying $\omega+\ddbar\vp>0$ and $\sup_{M}\vp=0$. Then there exists a constant $C$ depending only on $(M,\omega,J)$ such that
\begin{equation*}
\int_{M}|\vp| \, \omega^{n} = \int_{M}(-\vp)\omega^{n} \leq C.
\end{equation*}
\end{proposition}

\begin{proposition}[Proposition 3.1 of \cite{CTW16}]\label{Zero order estimate}
Let $\vp$ be a solution of (\ref{CMAE}). Then there exists a constant $C$ depending only on $(M,\omega,J)$, $\sup_{M}f$ such that
\begin{equation*}
\sup_{M}|\vp| \leq C.
\end{equation*}
\end{proposition}

Throughout this paper, we say a constant is uniform if it depends only on $(M,\omega,J)$, $\sup_{M}f$, $\sup_{M}|\de(f^{\frac{1}{n}})|_{g}$ and lower bound of $\nabla^{2}(f^{\frac{1}{n}})$. We also use Einstein notation convention. Sometimes, we will include the summation for clarity.

\section{First order estimate}
In this section, we prove the first order estimate. The proof is similar to \cite[Theorem 1]{Blocki09} in the K\"{a}hler case.

\begin{proposition}\label{First order estimate}
Let $\vp$ be a smooth solution of (\ref{CMAE}). Then there exists a constant $C$ depending only on $(M,\omega,J)$, $\sup_{M}f$, $\sup_{M}|\de(f^{\frac{1}{n}})|_{g}$ such that
\begin{equation*}
\sup_{M}|\de\vp|_{g} \leq C.
\end{equation*}
\end{proposition}

\begin{proof}
We consider the following quantity
\begin{equation*}
Q = \log|\de\vp|_{g}^{2} + e^{-A\vp},
\end{equation*}
where $A$ is a constant to be determined later. Let $x_{0}$ be the maximum point of $Q$ and $\{e_{i}\}_{i=1}^{n}$ be a local $g$-unitary frame for $T_{\mathbb{C}}^{(1,0)}M$ in a neighbourhood of $x_{0}$ such that
\begin{equation}\label{First order estimate equation 11}
\ti{g}_{i\ov{j}}(x_{0}) = \delta_{ij}\ti{g}_{i\ov{i}}(x_{0}).
\end{equation}

To prove Proposition \ref{First order estimate}, it suffices to prove that $|\de\vp|_{g}^{2}(x_{0})\leq C$. Without loss of generality, we assume that $|\de\vp|_{g}^{2}(x_{0})>1$. By the maximum principle, at $x_{0}$, we have
\begin{equation}\label{First order estimate equation 1}
\begin{split}
0 \geq {} & \ti{g}^{i\ov{i}}Q_{i\ov{i}} \\
    =  {} & \frac{\ti{g}^{i\ov{i}}(|\de\vp|_{g}^{2})_{i\ov{i}}}{|\de\vp|_{g}^{2}}
            -\frac{\ti{g}^{i\ov{i}}|(|\de\vp|_{g}^{2})_{i}|^{2}}{|\de\vp|_{g}^{4}}+\ti{g}^{i\ov{i}}(e^{-A\vp})_{i\ov{i}}.
\end{split}
\end{equation}
Now we estimate each term in (\ref{First order estimate equation 1}). For the first term of (\ref{First order estimate equation 1}), using (\ref{Commutation formula}) and (\ref{Expression of tilde g}), we compute
\begin{equation}\label{First order estimate equation 2}
\begin{split}
     {} & \ti{g}^{i\ov{i}}(|\de\vp|_{g}^{2})_{i\ov{i}} \\
   = {} & \sum_{k}\ti{g}^{i\ov{i}}(|\vp_{ik}|^{2}+|\vp_{i\ov{k}}|^{2})
          +\sum_{k}\ti{g}^{i\ov{i}}(\vp_{\ov{k}i\ov{i}}\vp_{k}+\vp_{ki\ov{i}}\vp_{\ov{k}}) \\
\geq {} & \sum_{k}\ti{g}^{i\ov{i}}(|\vp_{ik}|^{2}+|\vp_{i\ov{k}}|^{2})+2\textrm{Re}\left(\sum_{k}\ti{g}^{i\ov{i}}\vp_{i\ov{i}\ov{k}}\vp_{k}\right)
          -C|\de\vp|_{g}^{2}\sum_{i}\ti{g}^{i\ov{i}} \\
   = {} & \sum_{k}\ti{g}^{i\ov{i}}(|\vp_{ik}|^{2}+|\vp_{i\ov{k}}|^{2})
          +2\textrm{Re}\left(\sum_{k}\ti{g}^{i\ov{i}}(\ti{g}_{i\ov{i}})_{\ov{k}}\vp_{k}\right) \\
        & -2\textrm{Re}\left(\sum_{k}\ti{g}^{i\ov{i}}(S_{i\ov{i}}^{p}\vp_{p}+S_{i\ov{i}}^{\ov{p}}\vp_{\ov{p}})_{\ov{k}}\vp_{k}\right)
          -C|\de\vp|_{g}^{2}\sum_{i}\ti{g}^{i\ov{i}} \\
\geq {} & \sum_{k}\ti{g}^{i\ov{i}}(|\vp_{ik}|^{2}+|\vp_{i\ov{k}}|^{2})
          +2\textrm{Re}\left(\sum_{k}\ti{g}^{i\ov{i}}(\ti{g}_{i\ov{i}})_{\ov{k}}\vp_{k}\right) \\
        & -2\textrm{Re}\left(\sum_{k}\ti{g}^{i\ov{i}}(S_{i\ov{i}}^{p}\vp_{p\ov{k}}+S_{i\ov{i}}^{\ov{p}}\vp_{\ov{p}\ov{k}})\vp_{k}\right)
          -C|\de\vp|_{g}^{2}\sum_{i}\ti{g}^{i\ov{i}}.
\end{split}
\end{equation}
To deal with the third order term in (\ref{First order estimate equation 2}), we differentiate (covariantly) the logarithm of (\ref{CMAE})
\begin{equation*}
\log\frac{\det\ti{g}}{\det g} = n\log(f^{\frac{1}{n}}),
\end{equation*}
and we obtain
\begin{equation}\label{First order estimate equation 13}
\ti{g}^{i\ov{i}}(\ti{g}_{i\ov{i}})_{\ov{k}} = \frac{n(f^{\frac{1}{n}})_{\ov{k}}}{f^{\frac{1}{n}}}.
\end{equation}
By the arithmetic-geometric mean inequality, it is clear that
\begin{equation}\label{Arithmetic-geometric mean inequality}
\frac{1}{f^{\frac{1}{n}}} = \left(\prod_{i}\ti{g}^{i\ov{i}}\right)^{\frac{1}{n}} \leq \frac{1}{n}\sum_{i}\ti{g}^{i\ov{i}}.
\end{equation}
Combining (\ref{First order estimate equation 13}) and (\ref{Arithmetic-geometric mean inequality}), we have
\begin{equation}\label{First order estimate equation 12}
2\textrm{Re}\left(\sum_{k}\ti{g}^{i\ov{i}}(\ti{g}_{i\ov{i}})_{\ov{k}}\vp_{k}\right)
  =  2\textrm{Re}\left(\sum_{k}\frac{n(f^{\frac{1}{n}})_{\ov{k}}\vp_{k}}{f^{\frac{1}{n}}}\right)
\geq -C|\de\vp|_{g}\sum_{i}\ti{g}^{i\ov{i}}.
\end{equation}
Substituting (\ref{First order estimate equation 12}) into (\ref{First order estimate equation 2}), we see that
\begin{equation}\label{First order estimate equation 1 term 1}
\begin{split}
\ti{g}^{i\ov{i}}(|\de\vp|_{g}^{2})_{i\ov{i}}
\geq {} & \sum_{k}\ti{g}^{i\ov{i}}(|\vp_{ik}|^{2}+|\vp_{i\ov{k}}|^{2})-C|\de\vp|_{g}^{2}\sum_{i}\ti{g}^{i\ov{i}} \\
        & -2\textrm{Re}\left(\sum_{k}\ti{g}^{i\ov{i}}(S_{i\ov{i}}^{p}\vp_{p\ov{k}}+S_{i\ov{i}}^{\ov{p}}\vp_{\ov{p}\ov{k}})\vp_{k}\right),
\end{split}
\end{equation}
where we used $|\de\vp|_{g}^{2}(x_{0})>1$.

For the second term of (\ref{First order estimate equation 1}), using $Q_{i}(x_{0})=0$, it is clear that
\begin{equation}\label{First order estimate equation 1 term 2}
-\frac{\ti{g}^{i\ov{i}}|(|\de\vp|_{g}^{2})_{i}|^{2}}{|\de\vp|_{g}^{4}}
= -A^{2}e^{-2A\vp}\ti{g}^{i\ov{i}}|\vp_{i}|^{2}.
\end{equation}

For the third term of (\ref{First order estimate equation 1}), by (\ref{Expression of tilde g}), we see that
\begin{equation}\label{First order estimate equation 3}
\begin{split}
\ti{g}^{i\ov{i}}\vp_{i\ov{i}}
= {} & \ti{g}^{i\ov{i}}(\ti{g}_{i\ov{i}}-g_{i\ov{i}}-S_{i\ov{i}}^{p}\vp_{p}-S_{i\ov{i}}^{\ov{p}}\vp_{\ov{p}}) \\
= {} & n-\sum_{i}\ti{g}^{i\ov{i}}-\ti{g}^{i\ov{i}}(S_{i\ov{i}}^{p}\vp_{p}+S_{i\ov{i}}^{\ov{p}}\vp_{\ov{p}}),
\end{split}
\end{equation}
which implies
\begin{equation}\label{First order estimate equation 1 term 3}
\begin{split}
\ti{g}^{i\ov{i}}(e^{-A\vp})_{i\ov{i}}
= {} & Ae^{-A\vp}\sum_{i}\ti{g}^{i\ov{i}}+A^{2}e^{-A\vp}\ti{g}^{i\ov{i}}|\vp_{i}|^{2} \\
     & +Ae^{-A\vp}\ti{g}^{i\ov{i}}(S_{i\ov{i}}^{p}\vp_{p}+S_{i\ov{i}}^{\ov{p}}\vp_{\ov{p}})-Ane^{-A\vp}.
\end{split}
\end{equation}

Substituting (\ref{First order estimate equation 1 term 1}), (\ref{First order estimate equation 1 term 2}) and (\ref{First order estimate equation 1 term 3}) into (\ref{First order estimate equation 1}), we obtain
\begin{equation}\label{First order estimate equation 4}
\begin{split}
0 \geq {} & \frac{\sum_{k}\ti{g}^{i\ov{i}}(|\vp_{ik}|^{2}+|\vp_{i\ov{k}}|^{2})}{|\de\vp|_{g}^{2}}
            +(A^{2}e^{-A\vp}-A^{2}e^{-2A\vp})\ti{g}^{i\ov{i}}|\vp_{i}|^{2} \\
          & -2\textrm{Re}\left(\frac{\sum_{k}\ti{g}^{i\ov{i}}(S_{i\ov{i}}^{p}\vp_{p\ov{k}}+S_{i\ov{i}}^{\ov{p}}\vp_{\ov{p}\ov{k}})\vp_{k}}
            {|\de\vp|_{g}^{2}}\right)+Ae^{-A\vp}\ti{g}^{i\ov{i}}(S_{i\ov{i}}^{p}\vp_{p}+S_{i\ov{i}}^{\ov{p}}\vp_{\ov{p}}) \\[1mm]
          & +(Ae^{-A\vp}-C)\sum_{i}\ti{g}^{i\ov{i}}-Ane^{-A\vp}.
\end{split}
\end{equation}
Using $Q_{p}(x_{0})=0$ and (\ref{Property of S}), it is clear that
\begin{equation}\label{First order estimate equation 10}
\begin{split}
0 = {} & -2\textrm{Re}\left(\ti{g}^{i\ov{i}}S_{i\ov{i}}^{p}Q_{p}\right) \\[2mm]
  = {} & -2\textrm{Re}\left(\frac{\sum_{k}\ti{g}^{i\ov{i}}S_{i\ov{i}}^{p}(\vp_{kp}\vp_{\ov{k}}+\vp_{\ov{k}p}\vp_{k})}
         {|\de\vp|_{g}^{2}}-Ae^{-A\vp}\ti{g}^{i\ov{i}}S_{i\ov{i}}^{p}\vp_{p}\right) \\
  = {} & -2\textrm{Re}\left(\frac{\sum_{k}\ti{g}^{i\ov{i}}(S_{i\ov{i}}^{p}\vp_{\ov{k}p}+S_{i\ov{i}}^{\ov{p}}\vp_{\ov{k}\ov{p}})\vp_{k}}
            {|\de\vp|_{g}^{2}}\right)+Ae^{-A\vp}\ti{g}^{i\ov{i}}(S_{i\ov{i}}^{p}\vp_{p}+S_{i\ov{i}}^{\ov{p}}\vp_{\ov{p}}) \\
  = {} & -2\textrm{Re}\left(\frac{\sum_{k}\ti{g}^{i\ov{i}}(S_{i\ov{i}}^{p}\vp_{p\ov{k}}+S_{i\ov{i}}^{\ov{p}}\vp_{\ov{p}\ov{k}})\vp_{k}}
            {|\de\vp|_{g}^{2}}\right)+Ae^{-A\vp}\ti{g}^{i\ov{i}}(S_{i\ov{i}}^{p}\vp_{p}+S_{i\ov{i}}^{\ov{p}}\vp_{\ov{p}}),
\end{split}
\end{equation}
where we used $\vp_{kp}=\vp_{pk}$ and $\vp_{\ov{k}p}=\vp_{p\ov{k}}$ (Levi-Civita connection) in the last equality. Substituting (\ref{First order estimate equation 10}) into (\ref{First order estimate equation 4}), we obtain
\begin{equation}\label{First order estimate equation 5}
\begin{split}
0 \geq {} & \frac{\sum_{k}\ti{g}^{i\ov{i}}(|\vp_{ik}|^{2}+|\vp_{i\ov{k}}|^{2})}{|\de\vp|_{g}^{2}}
            +(A^{2}e^{-A\vp}-A^{2}e^{-2A\vp})\ti{g}^{i\ov{i}}|\vp_{i}|^{2} \\
          & +(Ae^{-A\vp}-C)\sum_{i}\ti{g}^{i\ov{i}}-Ane^{-A\vp}.
\end{split}
\end{equation}
By the Cauchy-Schwarz inequality, we get
\begin{equation}\label{First order estimate equation 6}
\frac{\sum_{k}\ti{g}^{i\ov{i}}|\vp_{ik}|^{2}}{|\de\vp|_{g}^{2}}
\geq \frac{\sum_{i}\ti{g}^{i\ov{i}}|\sum_{k}\vp_{ki}\vp_{\ov{k}}|^{2}}{|\de\vp|_{g}^{4}}.
\end{equation}
Using $Q_{i}(x_{0})=0$, it follows that
\begin{equation*}
\sum_{k}(\vp_{ki}\vp_{\ov{k}}+\vp_{\ov{k}i}\vp_{k}) = Ae^{-A\vp}|\de\vp|_{g}^{2}\vp_{i}.
\end{equation*}
Combining this with (\ref{Expression of tilde g}) and (\ref{First order estimate equation 11}), for any $\ve\in(0,1)$, we have
\begin{equation*}
\begin{split}
\left|\sum_{k}\vp_{ki}\vp_{\ov{k}}\right|^{2}
 =   {} & \left|Ae^{-A\vp}|\de\vp|_{g}^{2}\vp_{i}
          -\sum_{k}(\ti{g}_{i\ov{k}}-g_{i\ov{k}}-S_{i\ov{k}}^{p}\vp_{p}-S_{i\ov{k}}^{\ov{p}}\vp_{\ov{p}})\vp_{k}\right|^{2} \\
 =   {} & \left|(Ae^{-A\vp}|\de\vp|_{g}^{2}-\ti{g}_{i\ov{i}}+1)\vp_{i}
          +\sum_{k}(S_{i\ov{k}}^{p}\vp_{p}+S_{i\ov{k}}^{\ov{p}}\vp_{\ov{p}})\vp_{k}\right|^{2} \\
\geq {} & (1-\ve)\left(Ae^{-A\vp}|\de\vp|_{g}^{2}-\ti{g}_{i\ov{i}}+1\right)^{2}|\vp_{i}|^{2}-\frac{C}{\ve}|\de\vp|_{g}^{4} \\
\geq {} & (1-\ve)\left(A^{2}e^{-2A\vp}|\de\vp|_{g}^{4}-2Ae^{-A\vp}|\de\vp|_{g}^{2}\ti{g}_{i\ov{i}}-2\ti{g}_{i\ov{i}}\right)|\vp_{i}|^{2}
          -\frac{C}{\ve}|\de\vp|_{g}^{4}.
\end{split}
\end{equation*}
Substituting this into (\ref{First order estimate equation 6}), we see that
\begin{equation}\label{First order estimate equation 7}
\begin{split}
\frac{\sum_{k}\ti{g}^{i\ov{i}}|\vp_{ik}|^{2}}{|\de\vp|_{g}^{2}}
\geq {} & (1-\ve)A^{2}e^{-2A\vp}\ti{g}^{i\ov{i}}|\vp_{i}|^{2}-2Ae^{-A\vp}-\frac{2}{|\de\vp|_{g}^{2}}-\frac{C}{\ve}\sum_{i}\ti{g}^{i\ov{i}} \\
\geq {} & (1-\ve)A^{2}e^{-2A\vp}\ti{g}^{i\ov{i}}|\vp_{i}|^{2}-\frac{C}{\ve}\sum_{i}\ti{g}^{i\ov{i}}-2Ae^{-A\vp}-2,
\end{split}
\end{equation}
where we used $|\de\vp|_{g}^{2}(x_{0})>1$ in the second inequality. Substituting (\ref{First order estimate equation 7}) into (\ref{First order estimate equation 5}), we obtain
\begin{equation*}
\begin{split}
0 \geq {} & (A^{2}e^{-A\vp}-\ve A^{2}e^{-2A\vp})\ti{g}^{i\ov{i}}|\vp_{i}|^{2}
            +\left(Ae^{-A\vp}-\frac{C_{0}}{\ve}\right)\sum_{i}\ti{g}^{i\ov{i}} \\
          & -A(n+2)e^{-A\vp}-2,
\end{split}
\end{equation*}
where $C_{0}$ is a constant depending only on $(M,\omega,J)$, $\sup_{M}f$ and $\sup_{M}|\de(f^{\frac{1}{n}})|_{g}$. Now we choose
\begin{equation*}
A = 2C_{0}+1 \text{~and~} \ve = \frac{e^{A\vp(x_{0})}}{2}.
\end{equation*}
Recalling $\sup_{M}\vp=0$, we see that
\begin{equation*}
A^{2}e^{-A\vp}-\ve A^{2}e^{-2A\vp} \geq \frac{1}{2}
\text{~and~} Ae^{-A\vp}-\frac{C_{0}}{\ve} \geq 1.
\end{equation*}
It then follows that
\begin{equation}\label{First order estimate equation 9}
\frac{1}{2}\ti{g}^{i\ov{i}}|\vp_{i}|^{2}+\sum_{i}\ti{g}^{i\ov{i}} \leq C.
\end{equation}
From $\sum_{i}\ti{g}^{i\ov{i}}\leq C$ and $\frac{\det\ti{g}}{\det g}\leq C$, we have $\ti{g}_{i\ov{i}}\leq C$ for each $i$. Combining this with (\ref{First order estimate equation 9}), we obtain $|\de\vp|_{g}^{2}(x_{0})\leq C$, as desired.
\end{proof}

For later use, we state the following lemma, which follows from Proposition \ref{First order estimate}, (\ref{First order estimate equation 1 term 1}) and (\ref{First order estimate equation 3}).

\begin{lemma}\label{Calculation 1}
There exists a uniform constant $C$ such that
\begin{equation*}
\ti{g}^{i\ov{i}}\vp_{i\ov{i}}
= n-\sum_{i}\ti{g}^{i\ov{i}}-\ti{g}^{i\ov{i}}(S_{i\ov{i}}^{p}\vp_{p}+S_{i\ov{i}}^{\ov{p}}\vp_{\ov{p}})
\end{equation*}
and
\begin{equation*}
\begin{split}
\ti{g}^{i\ov{i}}(|\de\vp|_{g}^{2})_{i\ov{i}}
\geq {} & \sum_{k}\ti{g}^{i\ov{i}}(|\vp_{ik}|^{2}+|\vp_{i\ov{k}}|^{2})-C\sum_{i}\ti{g}^{i\ov{i}} \\
        & -2\textrm{Re}\left(\sum_{k}\ti{g}^{i\ov{i}}(S_{i\ov{i}}^{p}\vp_{p\ov{k}}+S_{i\ov{i}}^{\ov{p}}\vp_{\ov{p}\ov{k}})\vp_{k}\right).
\end{split}
\end{equation*}
\end{lemma}

\section{Second order estimate}
In this section, we prove the following second order estimate.

\begin{proposition}\label{Second order estimate}
Let $\vp$ be a smooth solution of (\ref{CMAE}). Then there exists a constant $C$ depending only on $(M,\omega,J)$, $\sup_{M}f$, $\sup_{M}|\de(f^{\frac{1}{n}})|_{g}$ and lower bound of $\nabla^{2}(f^{\frac{1}{n}})$ such that
\begin{equation*}
\sup_{M}|\nabla^{2}\vp|_{g} \leq C,
\end{equation*}
where $\nabla$ is the Levi-Civita connection of $g$.
\end{proposition}

\subsection{Auxiliary function}
Let $\lambda_{1}(\nabla^{2}\vp)\geq\lambda_{2}(\nabla^{2}\vp)\geq\cdots\geq\lambda_{2n}(\nabla^{2}\vp)$ be the eigenvalues of $\nabla^{2}\vp$. Combining $\omega+\ddbar\vp>0$, Proposition \ref{First order estimate} and \cite[(2.4),(2.5)]{CTW16}, we see that
\begin{equation*}
\sum_{\alpha=1}^{2n}\lambda_{\alpha}(\nabla^{2}\vp) = \Delta\vp \geq -C,
\end{equation*}
where $\Delta$ is the Laplace-Beltrami operator of $g$. It then follows that
\begin{equation}\label{Second order estimate equation 3}
|\nabla^{2}\vp|_{g} \leq C\max(\lambda_{1}(\nabla^{2}\vp),0) + C.
\end{equation}
To prove Proposition \ref{Second order estimate}, it suffice to prove $\lambda_{1}(\nabla^{2}\vp)$ is uniformly bounded from above. Without loss of generality, we assume that $D=\{x\in M~|~\lambda_{1}(\nabla^{2}\vp)(x)>0\}$ is not empty. On this set, we define the following quantity
\begin{equation*}
Q = \log\lambda_{1}(\nabla^{2}\vp)+h_{1}(|\ti{\omega}|_{g}^{2})+h_{2}(|\de\vp|_{g}^{2})+e^{-A\vp},
\end{equation*}
where
\begin{equation*}
h_{1}(s) = -\frac{1}{3}\log(10M_{R}^{2}-s), \,\,
h_{2}(s) = -\frac{1}{3}\log(1+\sup_{M}|\de\vp|_{g}^{2}-s),
\end{equation*}
$M_{R}=\sup_{M}|\nabla^{2}\vp|_{g}+1$ and $A>1$ is a constant to be determined later. We need to verify the function $h_{1}(|\ti{\omega}|_{g}^{2})$ is well defined. Without loss of generality, we assume that
\begin{equation*}
M_{R} \gg 1.
\end{equation*}
It then follows that
\begin{equation*}
|\ti{\omega}|_{g}^{2} \leq 2n+2|\de\dbar\vp|_{g}^{2} \leq 2n+4|\nabla^{2}\vp|_{g}^{2}+C|\de\vp|_{g}^{2} \leq 5M_{R}^{2},
\end{equation*}
which implies that $h_{1}(|\ti{\omega}|_{g}^{2})$ is well defined. By direct calculation, we have
\begin{equation}\label{Property 1 of h}
h_{1}'' = 3(h_{1}')^{2}, \,\,
h_{2}'' = 3(h_{2}')^{2},
\end{equation}
and
\begin{equation}\label{Property 2 of h}
\frac{1}{30M_{R}^{2}} \leq h_{1}' \leq \frac{1}{15M_{R}^{2}}, \,\,
\frac{1}{C} \leq h_{2}' \leq C.
\end{equation}

Clearly, the function $Q$ is continuous on its domain $D$ and equal to $-\infty$ on $\de D$. Let $x_{0}$ be the maximum point of $Q$. Then we have $\lambda_{1}(\nabla^{2}\vp)(x_{0})>0$. Let $\{e_{i}\}_{i=1}^{n}$ be a local $g$-unitary frame for $T_{\mathbb{C}}^{(1,0)}M$ in a neighbourhood of $x_{0}$ such that
\begin{equation}\label{tilde g diagonal}
\ti{g}_{i\ov{j}}(x_{0}) = \delta_{ij}\ti{g}_{i\ov{i}}(x_{0})
\text{~and~}
\ti{g}_{1\ov{1}}(x_{0}) \geq \ti{g}_{2\ov{2}}(x_{0}) \geq \cdots \geq \ti{g}_{n\ov{n}}(x_{0}).
\end{equation}
Since $(M,\omega,J)$ is almost Hermitian, there exists a coordinate system $(U;\{x^{\alpha}\}_{\alpha=1}^{2n})$ centered at $x_{0}$ such that it holds at $x_{0}$,
\begin{equation}\label{Property of coordinate system}
g_{\alpha\beta} = \delta_{\alpha\beta},~ \frac{\de g_{\alpha\beta}}{\de x^{\gamma}}=0 ~\text{~for $\alpha,\beta,\gamma=1,2,\cdots,2n$}
\end{equation}
and
\begin{equation}\label{Definition of e i}
J\de_{2i-1} = \de_{2i},~
e_{i} = \frac{1}{\sqrt{2}}(\de_{2i-1}-\sqrt{-1}\de_{2i})
~\text{~for $i=1,2,\cdots,n$}.
\end{equation}
We want to apply the maximum principle to the quantity $Q$ at $x_{0}$. However, $Q$ may be not smooth at $x_{0}$ when the eigenspace of $\lambda_{1}(\nabla^{2}\vp)$ has dimension great than $1$. To deal with this case, we apply a perturbation argument, as in \cite{CTW16}. For $\beta=1,2,\cdots,2n$, we write $V_{\beta}$ for the $g$-unit eigenvector of $\lambda_{\beta}(\nabla^{2}\vp)(x_{0})$ and denote the components of $V_{\beta}$ by $(V_{\beta}^{1},V_{\beta}^{2},\cdots,V_{\beta}^{2n})$. Next we extend $V_{\beta}$ to be vector fields near $x_{0}$ by taking the components to be constant and define a local endomorphism $\Phi_{\beta}^{\alpha}$ by
\begin{equation*}
\Phi_{\beta}^{\alpha} = g^{\alpha\gamma}\vp_{\gamma\beta}-g^{\alpha\gamma}B_{\gamma\beta}, \,
B_{\alpha\beta} = \delta_{\alpha\beta}-V_{1}^{\alpha}V_{1}^{\beta}.
\end{equation*}
Let $\lambda_{1}(\Phi)\geq\lambda_{2}(\Phi)\geq\cdots\geq\lambda_{2n}(\Phi)$ be the eigenvalues of $\Phi$. It follows that the vector $V_{\beta}(x_{0})$ is still the eigenvector of $\lambda_{\beta}(\Phi)(x_{0})$. By the definition of $\Phi$, at $x_{0}$, we have $\lambda_{1}(\Phi)>\lambda_{2}(\Phi)$, which implies the eigenspace of $\Phi$ corresponding to $\lambda_{1}(\Phi)$ has dimension $1$. Then $\lambda_{1}(\Phi)$ is smooth near $x_{0}$. In a neighborhood of $x_{0}$, we consider the perturbed quantity $\hat{Q}$ defined by
\begin{equation*}
\hat{Q} = \log\lambda_{1}(\Phi)+h_{1}(|\ti{\omega}|_{g}^{2})+h_{2}(|\de\vp|_{g}^{2})+e^{-A\vp}.
\end{equation*}
Since $\lambda_{1}(\Phi)(x_{0})=\lambda_{1}(\nabla^{2}\vp)(x_{0})$ and $\lambda_{1}(\Phi)\leq\lambda_{1}(\nabla^{2}\vp)$ near $x_{0}$, $\hat{Q}$ still attains a maximum at $x_{0}$. For convenience, we use $\lambda_{\beta}$ to denote $\lambda_{\beta}(\Phi)$ in the following argument.

On the other hand, by (\ref{Second order estimate equation 3}) and the definitions of $Q$, $\hat{Q}$ and $x_{0}$, it is clear that
\begin{equation}\label{Second order estimate equation 4}
\lambda_{1}(x_{0}) \leq M_{R} \leq C_{A}\lambda_{1}(x_{0}),
\end{equation}
where $M_{R}=\sup_{M}|\nabla^{2}\vp|_{g}+1$ and $C_{A}$ denotes a uniform constant depending on $A$. Without loss of generality, we assume that $\lambda_{1}(x_{0})\gg1$ in the following argument.

\subsection{Lower bound of $\ti{g}^{i\ov{i}}\hat{Q}_{i\ov{i}}$}
In this subsection, our aim is to obtain a lower bound of $\ti{g}^{i\ov{i}}\hat{Q}_{i\ov{i}}$ at $x_{0}$. First, we compute $\ti{g}^{i\ov{i}}(\lambda_{1})_{i\ov{i}}$ and $\ti{g}^{i\ov{i}}(|\ti{\omega}|_{g}^{2})_{i\ov{i}}$. Here we note that all the subscripts of a function denote the covariant derivatives with respect to the Levi-Civita connection $\nabla$.

\begin{lemma}\label{Calculation 2}
At $x_{0}$, we have
\begin{equation}\label{Calculation 2 equation 1}
\begin{split}
\ti{g}^{i\ov{i}}(\lambda_{1})_{i\ov{i}}
\geq {} & 2\sum_{\alpha>1}\frac{\ti{g}^{i\ov{i}}|\vp_{V_{\alpha}V_{1}i}|^{2}}{\lambda_{1}-\lambda_{\alpha}}
          +\sum_{p\neq q}\ti{g}^{p\ov{p}}\ti{g}^{q\ov{q}}|(\ti{g}_{p\ov{q}})_{V_{1}}|^{2}-C\lambda_{1}\sum_{i}\ti{g}^{i\ov{i}} \\
        & -\ti{g}^{i\ov{i}}\left(S_{i\ov{i}}^{p}\vp_{V_{1}V_{1}p}+S_{i\ov{i}}^{\ov{p}}\vp_{V_{1}V_{1}\ov{p}}\right)
\end{split}
\end{equation}
and
\begin{equation}\label{Calculation 2 equation 2}
\begin{split}
\ti{g}^{i\ov{i}}(|\ti{\omega}|_{g}^{2})
\geq {} & 2\sum_{k,l}\ti{g}^{i\ov{i}}|(\ti{g}_{k\ov{l}})_{i}|^{2}-CM_{R}^{2}\sum_{i}\ti{g}^{i\ov{i}} \\
        & -2\sum_{k}\ti{g}_{k\ov{k}}\ti{g}^{i\ov{i}}\left(S_{i\ov{i}}^{p}(\ti{g}_{k\ov{k}})_{p}+S_{i\ov{i}}^{\ov{p}}(\ti{g}_{k\ov{k}})_{\ov{p}}\right),
\end{split}
\end{equation}
where $M_{R}=\sup_{M}|\nabla^{2}\vp|_{g}+1$.
\end{lemma}

\begin{proof}
First, let us recall the elementary formulas (see \cite[Lemma 5.7]{CTW16}), holding at $x_{0}$,
\begin{equation}\label{Elementary formulas}
\begin{split}
\lambda_{1}^{\alpha\beta} {} & := \frac{\partial \lambda_{1}}{\partial \Phi^{\alpha}_{\beta}}=V_{1}^{\alpha}V_{1}^{\beta}, \\
\lambda_{1}^{\alpha\beta,\gamma\delta} {} & := \frac{\partial^{2}\lambda_{1}}{\partial \Phi^{\alpha}_{\beta}\partial\Phi^{\gamma}_{\delta}}
= \sum_{\mu>1}\frac{V_{1}^{\alpha}V_{\mu}^{\beta}V_{\mu}^{\gamma}V_{1}^{\delta}+V_{\mu}^{\alpha}V_{1}^{\beta}V_{1}^{\gamma}V_{\mu}^{\delta}}
  {\lambda_{1}-\lambda_{\mu}}.
\end{split}
\end{equation}

For (\ref{Calculation 2 equation 1}), using (\ref{Elementary formulas}) and (\ref{Property of coordinate system}), we compute
\begin{equation}\label{Calculation 2 equation 1.1}
\begin{split}
\ti{g}^{i\ov{i}}(\lambda_{1})_{i\ov{i}}
 =   {} & \ti{g}^{i\ov{i}}\lambda_{1}^{\alpha\beta,\gamma\delta}(\Phi_{\delta}^{\gamma})_{i}(\Phi_{\beta}^{\alpha})_{\ov{i}}
          +\ti{g}^{i\ov{i}}\lambda_{1}^{\alpha\beta}(\Phi_{\beta}^{\alpha})_{i\ov{i}} \\
 =   {} & \ti{g}^{i\ov{i}}\lambda_{1}^{\alpha\beta,\gamma\delta}\vp_{\gamma\delta i}\vp_{\alpha\beta\ov{i}}
          +\ti{g}^{i\ov{i}}\lambda_{1}^{\alpha\beta}\vp_{\alpha\beta i\ov{i}}
          -\ti{g}^{i\ov{i}}\lambda_{1}^{\alpha\beta}(B_{\alpha\beta})_{i\ov{i}} \\
\geq {} & 2\sum_{\alpha>1}\frac{\ti{g}^{i\ov{i}}|\vp_{V_{\alpha}V_{1}i}|^{2}}{\lambda_{1}-\lambda_{\alpha}}
          +\ti{g}^{i\ov{i}}\vp_{V_{1}V_{1}i\ov{i}}-C\sum_{i}\ti{g}^{i\ov{i}} \\
\geq {} & 2\sum_{\alpha>1}\frac{\ti{g}^{i\ov{i}}|\vp_{V_{\alpha}V_{1}i}|^{2}}{\lambda_{1}-\lambda_{\alpha}}
          +\ti{g}^{i\ov{i}}\vp_{i\ov{i}V_{1}V_{1}}-C\lambda_{1}\sum_{i}\ti{g}^{i\ov{i}},
\end{split}
\end{equation}
where we used (\ref{Commutation formula}) and (\ref{Second order estimate equation 3}) in the last inequality. Recalling (\ref{Expression of tilde g}) and using (\ref{Commutation formula}) again, we see that
\begin{equation}\label{Calculation 2 equation 1.2}
\begin{split}
\ti{g}^{i\ov{i}}\vp_{i\ov{i}V_{1}V_{1}}
   = {} & \ti{g}^{i\ov{i}}\left(\ti{g}_{i\ov{i}}-g_{i\ov{i}}-S_{i\ov{i}}^{p}\vp_{p}-S_{i\ov{i}}^{\ov{p}}\vp_{\ov{p}}\right)_{V_{1}V_{1}} \\[1mm]
\geq {} & \ti{g}^{i\ov{i}}(\ti{g}_{i\ov{i}})_{V_{1}V_{1}}
          -\ti{g}^{i\ov{i}}\left(S_{i\ov{i}}^{p}\vp_{pV_{1}V_{1}}+S_{i\ov{i}}^{\ov{p}}\vp_{\ov{p}V_{1}V_{1}}\right)
          -C\lambda_{1}\sum_{i}\ti{g}^{i\ov{i}} \\
\geq {} & \ti{g}^{i\ov{i}}(\ti{g}_{i\ov{i}})_{V_{1}V_{1}}
          -\ti{g}^{i\ov{i}}\left(S_{i\ov{i}}^{p}\vp_{V_{1}V_{1}p}+S_{i\ov{i}}^{\ov{p}}\vp_{V_{1}V_{1}\ov{p}}\right)
          -C\lambda_{1}\sum_{i}\ti{g}^{i\ov{i}}.
\end{split}
\end{equation}
Applying $\nabla_{V_{1}}$ to the logarithm of (\ref{CMAE}), it follows that
\begin{equation}\label{Calculation 2 equation 1.3}
\ti{g}^{i\ov{j}}(\ti{g}_{i\ov{j}})_{V_{1}} = \frac{n(f^{\frac{1}{n}})_{V_{1}}}{f^{\frac{1}{n}}}.
\end{equation}
Applying $\nabla_{V_{1}}$ again, at $x_{0}$, we have
\begin{equation}\label{Calculation 2 equation 1.4}
\ti{g}^{i\ov{i}}(\ti{g}_{i\ov{i}})_{V_{1}V_{1}}
= \ti{g}^{p\ov{p}}\ti{g}^{q\ov{q}}|(\ti{g}_{p\ov{q}})_{V_{1}}|^{2}
  +\frac{n(f^{\frac{1}{n}})_{V_{1}V_{1}}}{f^{\frac{1}{n}}}
  -\frac{n|(f^{\frac{1}{n}})_{V_{1}}|^{2}}{f^{\frac{2}{n}}}.
\end{equation}
Substituting (\ref{Calculation 2 equation 1.3}) into (\ref{Calculation 2 equation 1.4}) and using the Cauchy-Schwarz inequality, we compute
\begin{equation*}
\begin{split}
\ti{g}^{i\ov{i}}(\ti{g}_{i\ov{i}})_{V_{1}V_{1}}
  =  {} & \sum_{p\neq q}\ti{g}^{p\ov{p}}\ti{g}^{q\ov{q}}|(\ti{g}_{p\ov{q}})_{V_{1}}|^{2}
          +\sum_{p}\left|\ti{g}^{p\ov{p}}(\ti{g}_{p\ov{p}})_{V_{1}}\right|^{2} \\
        & +\frac{n(f^{\frac{1}{n}})_{V_{1}V_{1}}}{f^{\frac{1}{n}}}
          -\frac{1}{n}\left|\sum_{p}\ti{g}^{p\ov{p}}(\ti{g}_{p\ov{p}})_{V_{1}}\right|^{2} \\
\geq {} & \sum_{p\neq q}\ti{g}^{p\ov{p}}\ti{g}^{q\ov{q}}|(\ti{g}_{p\ov{q}})_{V_{1}}|^{2}
          +\frac{n(f^{\frac{1}{n}})_{V_{1}V_{1}}}{f^{\frac{1}{n}}}.
\end{split}
\end{equation*}
Combining this with (\ref{Arithmetic-geometric mean inequality}), it is clear that
\begin{equation}\label{Calculation 2 equation 1.5}
\ti{g}^{i\ov{i}}(\ti{g}_{i\ov{i}})_{V_{1}V_{1}}
\geq \sum_{p\neq q}\ti{g}^{p\ov{p}}\ti{g}^{q\ov{q}}|(\ti{g}_{p\ov{q}})_{V_{1}}|^{2}-C\sum_{i}\ti{g}^{i\ov{i}}.
\end{equation}
Then the inequality (\ref{Calculation 2 equation 1}) follows from (\ref{Calculation 2 equation 1.1}), (\ref{Calculation 2 equation 1.2}) and (\ref{Calculation 2 equation 1.5}).

For (\ref{Calculation 2 equation 2}), a direct calculation shows that
\begin{equation}\label{Calculation 2 equation 2.1}
\ti{g}^{i\ov{i}}(|\ti{\omega}|_{g}^{2})
= 2\sum_{k,l}\ti{g}^{i\ov{i}}\left|(\ti{g}_{k\ov{l}})_{i}\right|^{2}+2\sum_{k}\ti{g}_{k\ov{k}}\ti{g}^{i\ov{i}}(\ti{g}_{k\ov{k}})_{i\ov{i}}.
\end{equation}
By (\ref{Expression of tilde g}) and (\ref{Commutation formula}), for each $k=1,2,\cdots,n$, we have
\begin{equation}\label{Calculation 2 equation 2.2}
\begin{split}
\ti{g}^{i\ov{i}}(\ti{g}_{k\ov{k}})_{i\ov{i}}
   = {} & \ti{g}^{i\ov{i}}\left(g_{k\ov{k}}+\vp_{k\ov{k}}+S_{k\ov{k}}^{p}\vp_{p}+S_{k\ov{k}}^{\ov{p}}\vp_{\ov{p}}\right)_{i\ov{i}} \\
\geq {} & \ti{g}^{i\ov{i}}\vp_{k\ov{k}i\ov{i}}
          +\ti{g}^{i\ov{i}}\left(S_{k\ov{k}}^{p}\vp_{pi\ov{i}}+S_{k\ov{k}}^{\ov{p}}\vp_{\ov{p}i\ov{i}}\right)
          -C\lambda_{1}\sum_{i}\ti{g}^{i\ov{i}} \\
\geq {} & \ti{g}^{i\ov{i}}\vp_{i\ov{i}k\ov{k}}
          +\ti{g}^{i\ov{i}}(\ti{g}_{i\ov{i}})_{p}S_{k\ov{k}}^{p}+\ti{g}^{i\ov{i}}(\ti{g}_{i\ov{i}})_{\ov{p}}S_{k\ov{k}}^{\ov{p}}
          -C\lambda_{1}\sum_{i}\ti{g}^{i\ov{i}}.
\end{split}
\end{equation}
By the similar calculations of (\ref{Calculation 2 equation 1.3}) and (\ref{Calculation 2 equation 1.5}), it follows from (\ref{Arithmetic-geometric mean inequality}) that
\begin{equation}\label{Calculation 2 equation 2.3}
|\ti{g}^{i\ov{i}}(\ti{g}_{i\ov{i}})_{p}|
  =  \frac{|n(f^{\frac{1}{n}})_{p}|}{f^{\frac{1}{n}}}
\geq -C\sum_{i}\ti{g}^{i\ov{i}}
\end{equation}
and
\begin{equation}\label{Calculation 2 equation 2.4}
\ti{g}^{i\ov{i}}(\ti{g}_{i\ov{i}})_{k\ov{k}}
\geq \sum_{p\neq q}\ti{g}^{p\ov{p}}\ti{g}^{q\ov{q}}|(\ti{g}_{p\ov{q}})_{k}|^{2}-C\sum_{i}\ti{g}^{i\ov{i}}.
\end{equation}
For the first term of (\ref{Calculation 2 equation 2.2}), using (\ref{Expression of tilde g}), (\ref{Commutation formula}) and (\ref{Calculation 2 equation 2.4}), we compute
\begin{equation}\label{Calculation 2 equation 2.5}
\begin{split}
\ti{g}^{i\ov{i}}\vp_{i\ov{i}k\ov{k}}
   = {} & \ti{g}^{i\ov{i}}\left(\ti{g}_{i\ov{i}}-g_{i\ov{i}}-S_{i\ov{i}}^{p}\vp_{p}-S_{i\ov{i}}^{\ov{p}}\vp_{\ov{p}}\right)_{k\ov{k}} \\
\geq {} & \ti{g}^{i\ov{i}}(\ti{g}_{i\ov{i}})_{k\ov{k}}
          -\ti{g}^{i\ov{i}}\left(S_{i\ov{i}}^{p}\vp_{pk\ov{k}}+S_{i\ov{i}}^{\ov{p}}\vp_{\ov{p}k\ov{k}}\right)
          -C\lambda_{1}\sum_{i}\ti{g}^{i\ov{i}} \\
\geq {} & \ti{g}^{i\ov{i}}(\ti{g}_{i\ov{i}})_{k\ov{k}}
          -\ti{g}^{i\ov{i}}\left(S_{i\ov{i}}^{p}\vp_{k\ov{k}p}+S_{i\ov{i}}^{\ov{p}}\vp_{k\ov{k}\ov{p}}\right)
          -C\lambda_{1}\sum_{i}\ti{g}^{i\ov{i}} \\
\geq {} & \sum_{p\neq q}\ti{g}^{p\ov{p}}\ti{g}^{q\ov{q}}|(\ti{g}_{p\ov{q}})_{k}|^{2}
          -\ti{g}^{i\ov{i}}S_{i\ov{i}}^{p}(\ti{g}_{k\ov{k}})_{p}-\ti{g}^{i\ov{i}}S_{i\ov{i}}^{\ov{p}}(\ti{g}_{k\ov{k}})_{\ov{p}}
          -C\lambda_{1}\sum_{i}\ti{g}^{i\ov{i}}.
\end{split}
\end{equation}
For the second and third terms of (\ref{Calculation 2 equation 2.2}), by (\ref{Calculation 2 equation 2.3}), we get
\begin{equation}\label{Calculation 2 equation 2.6}
\ti{g}^{i\ov{i}}(\ti{g}_{i\ov{i}})_{p}S_{k\ov{k}}^{p}+\ti{g}^{i\ov{i}}(\ti{g}_{i\ov{i}})_{\ov{p}}S_{k\ov{k}}^{\ov{p}}
\geq -C\sum_{i}\ti{g}^{i\ov{i}}.
\end{equation}
Substituting (\ref{Calculation 2 equation 2.5}) and (\ref{Calculation 2 equation 2.6}) into (\ref{Calculation 2 equation 2.2}), it is clear that
\begin{equation}\label{Calculation 2 equation 2.7}
\begin{split}
\ti{g}^{i\ov{i}}(\ti{g}_{k\ov{k}})_{i\ov{i}}
\geq {} & \sum_{p\neq q}\ti{g}^{p\ov{p}}\ti{g}^{q\ov{q}}|(\ti{g}_{p\ov{q}})_{k}|^{2}
          -\ti{g}^{i\ov{i}}S_{i\ov{i}}^{p}(\ti{g}_{k\ov{k}})_{p} \\
        & -\ti{g}^{i\ov{i}}S_{i\ov{i}}^{\ov{p}}(\ti{g}_{k\ov{k}})_{\ov{p}}
          -C\lambda_{1}\sum_{i}\ti{g}^{i\ov{i}} \\
\geq {} & -\ti{g}^{i\ov{i}}S_{i\ov{i}}^{p}(\ti{g}_{k\ov{k}})_{p}
          -\ti{g}^{i\ov{i}}S_{i\ov{i}}^{\ov{p}}(\ti{g}_{k\ov{k}})_{\ov{p}}
          -C\lambda_{1}\sum_{i}\ti{g}^{i\ov{i}}.
\end{split}
\end{equation}
By (\ref{Second order estimate equation 3}), we have $0<\ti{g}_{k\ov{k}}\leq C\lambda_{1}$. Combining this with (\ref{Calculation 2 equation 2.1}) and (\ref{Calculation 2 equation 2.7}), we see that
\begin{equation*}
\begin{split}
\ti{g}^{i\ov{i}}(|\ti{\omega}|_{g}^{2})
\geq {} & 2\sum_{k,l}\ti{g}^{i\ov{i}}|(\ti{g}_{k\ov{l}})_{i}|^{2}-C\lambda_{1}^{2}\sum_{i}\ti{g}^{i\ov{i}} \\
        & -2\sum_{k}\ti{g}_{k\ov{k}}\ti{g}^{i\ov{i}}\left(S_{i\ov{i}}^{p}(\ti{g}_{k\ov{k}})_{p}+S_{i\ov{i}}^{\ov{p}}(\ti{g}_{k\ov{k}})_{\ov{p}}\right).
\end{split}
\end{equation*}
Using $\lambda_{1}\leq M_{R}$, we obtain the inequality (\ref{Calculation 2 equation 2}).
\end{proof}

\begin{lemma}\label{Main inequality}
At $x_{0}$, we have
\begin{equation}\label{Main inequality equation 1}
\begin{split}
0 \geq {} & \ti{g}^{i\ov{i}}\hat{Q}_{i\ov{i}} \\
  \geq {} & 2\sum_{\alpha>1}\frac{\ti{g}^{i\ov{i}}|\vp_{V_{\alpha}V_{1}i}|^{2}}{\lambda_{1}(\lambda_{1}-\lambda_{\alpha})}
            +\sum_{p\neq q}\frac{\ti{g}^{p\ov{p}}\ti{g}^{q\ov{q}}|(\ti{g}_{p\ov{q}})_{V_{1}}|^{2}}{\lambda_{1}}
            -\frac{\ti{g}^{i\ov{i}}|\vp_{V_{1}V_{1}i}|^{2}}{\lambda_{1}^{2}} \\
          & +2h_{1}'\sum_{k,l}\ti{g}^{i\ov{i}}|(\ti{g}_{k\ov{l}})_{i}|^{2}
            +h_{2}'\sum_{k}\ti{g}^{i\ov{i}}(|\vp_{ik}|^{2}+|\vp_{i\ov{k}}|^{2}) \\
          & +h_{1}''\ti{g}^{i\ov{i}}|(|\ti{\omega}|_{g}^{2})_{i}|^{2}
            +h_{2}''\ti{g}^{i\ov{i}}|(|\de\vp|_{g}^{2})_{i}|^{2}
            +A^{2}e^{-A\vp}\ti{g}^{i\ov{i}}|\vp_{i}|^{2} \\[3mm]
          & +(Ae^{-A\vp}-C)\sum_{i}\ti{g}^{i\ov{i}}-Ane^{-A\vp}.
\end{split}
\end{equation}
\end{lemma}

\begin{proof}
For convenience, we use $J$ to denote the right hand side of (\ref{Main inequality equation 1}). It suffice to prove $\ti{g}^{i\ov{i}}\hat{Q}_{i\ov{i}} \geq J$ at $x_{0}$. Combining Lemma \ref{Calculation 1}, \ref{Calculation 2} and $h_{1}'\leq\frac{1}{15M_{R}^{2}}$ (see (\ref{Property 2 of h})), we obtain
\begin{equation*}
\ti{g}^{i\ov{i}}\hat{Q}_{i\ov{i}} \geq J+\ti{J},
\end{equation*}
where
\begin{equation}\label{Main inequality equation 2}
\begin{split}
\ti{J}
= {} & -\frac{\ti{g}^{i\ov{i}}(S_{i\ov{i}}^{p}\vp_{V_{1}V_{1}p}+S_{i\ov{i}}^{\ov{p}}\vp_{V_{1}V_{1}\ov{p}})}{\lambda_{1}}
       -2h_{1}'\sum_{k}\ti{g}_{k\ov{k}}\ti{g}^{i\ov{i}}
       \left(S_{i\ov{i}}^{p}(\ti{g}_{k\ov{k}})_{p}+S_{i\ov{i}}^{\ov{p}}(\ti{g}_{k\ov{k}})_{\ov{p}}\right) \\
     & -2h_{2}'\textrm{Re}\left(\sum_{k}\ti{g}^{i\ov{i}}
       (S_{i\ov{i}}^{p}\vp_{p\ov{k}}+S_{i\ov{i}}^{\ov{p}}\vp_{\ov{p}\ov{k}})\vp_{k}\right)
       +Ae^{-A\vp}\ti{g}^{i\ov{i}}(S_{i\ov{i}}^{p}\vp_{p}+S_{i\ov{i}}^{\ov{p}}\vp_{\ov{p}}).
\end{split}
\end{equation}
Using $\hat{Q}_{p}(x_{0})=0$ and the similar calculation of (\ref{First order estimate equation 10}), it then follows that
\begin{equation}\label{Main inequality equation 3}
\ti{J} = -2\textrm{Re}\left(\ti{g}^{i\ov{i}}S_{i\ov{i}}^{p}\hat{Q}_{p}\right) = 0,
\end{equation}
as required.
\end{proof}

\subsection{Proof of Proposition \ref{Second order estimate}}
In this subsection, we give the proof of Proposition \ref{Second order estimate}.

\subsection*{$\bullet$ Partial second order estimate}
We define
\begin{equation*}
I := \{ i\in\{1,\cdots,n\} ~|~ \ti{g}_{i\ov{i}}\geq A^{3}e^{-2A\vp}\ti{g}_{n\ov{n}} \text{~at $x_{0}$} \}.
\end{equation*}
Since $A>1$ and $\sup_{M}\vp=0$, we have $n\notin I$. The following lemma can be regarded as partial second order estimate.

\begin{lemma}\label{Partial second order estimate}
At $x_{0}$, we have
\begin{equation*}
\sum_{k}\sum_{i\notin I}\left(|\vp_{ik}|^{2}+|\vp_{i\ov{k}}|^{2}\right) \leq C_{A},
\end{equation*}
where $C_{A}$ is a uniform constant depending on $A$.
\end{lemma}

\begin{proof}
Using $\hat{Q}_{i}(x_{0})=0$ and the Cauchy-Schwarz inequality, for each $i=1,2,\cdots,n$, it is clear that
\begin{equation}\label{Partial second order estimate equation 1}
\begin{split}
\frac{\ti{g}^{i\ov{i}}|\vp_{V_{1}V_{1}i}|^{2}}{\lambda_{1}^{2}}
  =  {} & \ti{g}^{i\ov{i}}\left|h_{1}'(|\ti{\omega}|_{g}^{2})_{i}+h_{2}'(|\de\vp|_{g}^{2})_{i}-Ae^{-A\vp}\vp_{i}\right|^{2} \\
\leq {} & 3(h_{1}')^{2}\ti{g}^{i\ov{i}}|(|\ti{\omega}|_{g}^{2})_{i}|^{2}+3(h_{2}')^{2}\ti{g}^{i\ov{i}}|(|\de\vp|_{g}^{2})_{i}|^{2} \\[2mm]
        & +3A^{2}e^{-2A\vp}\ti{g}^{i\ov{i}}|\vp_{i}|^{2}.
\end{split}
\end{equation}
Combining this with Lemma \ref{Main inequality} and discarding some positive terms, we obtain
\begin{equation*}
\begin{split}
0 \geq {} & h_{2}'\sum_{k}\ti{g}^{i\ov{i}}\left(|\vp_{ik}|^{2}+|\vp_{i\ov{k}}|^{2}\right)
            +\left(h_{1}''-3(h_{1}')^{2}\right)\sum_{i}\ti{g}^{i\ov{i}}|(|\ti{\omega}|_{g}^{2})_{i}|^{2} \\
          & +\left(h_{2}''-3(h_{2}')^{2}\right)\sum_{i}\ti{g}^{i\ov{i}}|(|\de\vp|_{g}^{2})_{i}|^{2}-CA^{2}e^{-2A\vp}\sum_{i}\ti{g}^{i\ov{i}},
\end{split}
\end{equation*}
where we used $\sum_{i}\ti{g}^{i\ov{i}}\geq C^{-1}$ (see (\ref{Arithmetic-geometric mean inequality})). Using (\ref{Property 1 of h}), (\ref{Property 2 of h}), (\ref{tilde g diagonal}) and the definition of $I$, it is clear that
\begin{equation*}
\begin{split}
0 \geq {} & C^{-1}\sum_{k}\ti{g}^{i\ov{i}}\left(|\vp_{ik}|^{2}+|\vp_{i\ov{k}}|^{2}\right)-CA^{2}e^{-2A\vp}\sum_{i}\ti{g}^{i\ov{i}} \\
  \geq {} & C^{-1}\sum_{k}\sum_{i\notin I}\ti{g}^{i\ov{i}}\left(|\vp_{ik}|^{2}+|\vp_{i\ov{k}}|^{2}\right)-CA^{2}e^{-2A\vp}\ti{g}^{n\ov{n}} \\
  \geq {} & C^{-1}A^{-3}e^{2A\vp}\ti{g}^{n\ov{n}}\sum_{k}\sum_{i\notin I}\left(|\vp_{ik}|^{2}+|\vp_{i\ov{k}}|^{2}\right)
            -CA^{2}e^{-2A\vp}\ti{g}^{n\ov{n}},
\end{split}
\end{equation*}
as desired.
\end{proof}

Clearly, if $I=\emptyset$, then Proposition \ref{Second order estimate} follows from Lemma \ref{Partial second order estimate}. Hence, we assume $I\neq\emptyset$ in the following argument.

\subsection*{$\bullet$ Third order terms}
The key point is to deal with the "bad" third order term
\begin{equation}\label{Bad term K}
K := \frac{\ti{g}^{i\ov{i}}|\vp_{V_{1}V_{1}i}|^{2}}{\lambda_{1}^{2}}.
\end{equation}
For any $\ve\in(0,\frac{1}{3})$, we decompose the term $K$ into three parts as follows:
\begin{equation*}
\begin{split}
K
= {} & \sum_{i\in I}\frac{\ti{g}^{i\ov{i}}|\vp_{V_{1}V_{1}i}|^{2}}{\lambda_{1}^{2}}
+2\ve\sum_{i\notin I}\frac{\ti{g}^{i\ov{i}}|\vp_{V_{1}V_{1}i}|^{2}}{\lambda_{1}^{2}}
+(1-2\ve)\sum_{i\notin I}\frac{\ti{g}^{i\ov{i}}|\vp_{V_{1}V_{1}i}|^{2}}{\lambda_{1}^{2}} \\
=: {} & K_{1}+K_{2}+K_{3}.
\end{split}
\end{equation*}

\begin{lemma}\label{Bad term 1 and 2}
At $x_{0}$, we have
\begin{equation}\label{Bad term 1 and 2 equation 1}
\begin{split}
K_{1}+K_{2}
\leq {} & 3(h_{1}')^{2}\ti{g}^{i\ov{i}}|(|\ti{\omega}|_{g}^{2})_{i}|^{2}
          +3(h_{2}')^{2}\ti{g}^{i\ov{i}}|(|\de\vp|_{g}^{2})_{i}|^{2} \\[3mm]
        & +6\ve A^{2}e^{-2A\vp}\ti{g}^{i\ov{i}}|\vp_{i}|^{2}+C\sum_{i}\ti{g}^{i\ov{i}}.
\end{split}
\end{equation}
\end{lemma}

\begin{proof}
Using (\ref{Partial second order estimate equation 1}) and the definition of $I$, we obtain
\begin{equation}\label{Bad term 1 and 2 equation 2}
\begin{split}
\sum_{i\in I}\frac{\ti{g}^{i\ov{i}}|\vp_{V_{1}V_{1}i}|^{2}}{\lambda_{1}^{2}}
\leq {} & 3(h_{1}')^{2}\sum_{i\in I}\ti{g}^{i\ov{i}}|(|\ti{\omega}|_{g}^{2})_{i}|^{2}
              +3(h_{2}')^{2}\sum_{i\in I}\ti{g}^{i\ov{i}}|(|\de\vp|_{g}^{2})_{i}|^{2} \\
        & +3A^{2}e^{-2A\vp}\sum_{i\in I}\ti{g}^{i\ov{i}}|\vp_{i}|^{2} \\
\leq {} & 3(h_{1}')^{2}\sum_{i\in I}\ti{g}^{i\ov{i}}|(|\ti{\omega}|_{g}^{2})_{i}|^{2}
              +3(h_{2}')^{2}\sum_{i\in I}\ti{g}^{i\ov{i}}|(|\de\vp|_{g}^{2})_{i}|^{2} \\
        & +\frac{3n\sup_{M}|\de\vp|_{g}^{2}}{A}\ti{g}^{n\ov{n}}.
\end{split}
\end{equation}
By the similar calculation, it is clear that
\begin{equation}\label{Bad term 1 and 2 equation 3}
\begin{split}
2\ve\sum_{i\notin I}\frac{\ti{g}^{i\ov{i}}|\vp_{V_{1}V_{1}i}|^{2}}{\lambda_{1}^{2}}
\leq {} & 6\ve(h_{1}')^{2}\sum_{i\notin I}\ti{g}^{i\ov{i}}|(|\ti{\omega}|_{g}^{2})_{i}|^{2}
              +6\ve(h_{2}')^{2}\sum_{i\notin I}\ti{g}^{i\ov{i}}|(|\de\vp|_{g}^{2})_{i}|^{2} \\
        & +6\ve A^{2}e^{-2A\vp}\sum_{i\notin I}\ti{g}^{i\ov{i}}|\vp_{i}|^{2}.
\end{split}
\end{equation}
Combining (\ref{Bad term 1 and 2 equation 2}), (\ref{Bad term 1 and 2 equation 3}), $\ve\in(0,\frac{1}{3})$ and $A>1$, we obtain (\ref{Bad term 1 and 2 equation 1}).
\end{proof}

In order to deal with the term $K_{3}$, we define a local $(1,0)$ vector field by
\begin{equation}\label{Definition of tilde e 1}
\ti{e}_{1} = \frac{1}{\sqrt{2}}(V_{1}-\sqrt{-1}JV_{1}).
\end{equation}
At $x_{0}$, since $|\ti{e}_{1}|_{g}=|JV_{1}|_{g}=1$, we write
\begin{equation}\label{Expression of tilde e 1}
\ti{e}_{1} = \sum_{q}\nu_{q}e_{q}, \quad \sum_{q}|\nu_{q}|^{2}=1
\end{equation}
and
\begin{equation}\label{Expression of JV 1}
JV_{1} = \sum_{\alpha>1}\mu_{\alpha}V_{\alpha}, \quad \sum_{\alpha>1}\mu_{\alpha}^{2}=1,
\end{equation}
where we used the vector $JV_{1}$ is $g$-orthogonal to $V_{1}$.

\begin{lemma}\label{Bad term 3}
At $x_{0}$, if $\lambda_{1}\geq\frac{C_{A}}{\ve^{3}}$ for a uniform constant $C_{A}$ depending on $A$, then we have
\begin{equation*}
\begin{split}
K_{3}
\leq {} & 2\sum_{\alpha>1}\frac{\ti{g}^{i\ov{i}}|\vp_{V_{\alpha}V_{1}i}|^{2}}{\lambda_{1}(\lambda_{1}-\lambda_{\alpha})}
          +\sum_{p\neq q}\frac{\ti{g}^{p\ov{p}}\ti{g}^{q\ov{q}}|(\ti{g}_{p\ov{q}})_{V_{1}}|^{2}}{\lambda_{1}}
          +2h_{1}'\sum_{k,l}\ti{g}^{i\ov{i}}|(\ti{g}_{k\ov{l}})_{i}|^{2}+\frac{C}{\ve}\sum_{i}\ti{g}^{i\ov{i}}.
\end{split}
\end{equation*}
\end{lemma}

\begin{proof}
By (\ref{Definition of tilde e 1}), (\ref{Commutation formula}) and (\ref{Expression of tilde g}), we compute
\begin{equation}\label{Bad term 3 equation 1}
\begin{split}
\vp_{V_{1}V_{1}i}
= {} & \sqrt{2}\vp_{V_{1}\ov{\ti{e}}_{1}i}-\sqrt{-1}\vp_{V_{1}JV_{1}i} \\[2mm]
= {} & \sqrt{2}\sum_{q}\ov{\nu_{q}}\vp_{V_{1}\ov{q}i}-\sqrt{-1}\sum_{\alpha>1}\mu_{\alpha}\vp_{V_{1}V_{\alpha}i} \\
= {} & \sqrt{2}\sum_{q}\ov{\nu_{q}}\vp_{i\ov{q}V_{1}}-\sqrt{-1}\sum_{\alpha>1}\mu_{\alpha}\vp_{V_{\alpha}V_{1}i}+E \\
= {} & \sqrt{2}\sum_{q}\ov{\nu_{q}}(\ti{g}_{i\ov{q}})_{V_{1}}-\sqrt{-1}\sum_{\alpha>1}\mu_{\alpha}\vp_{V_{\alpha}V_{1}i}+E \\
= {} & \sqrt{2}\sum_{q\notin I}\ov{\nu_{q}}(\ti{g}_{i\ov{q}})_{V_{1}}+\sqrt{2}\sum_{q\in I}\ov{\nu_{q}}(\ti{g}_{i\ov{q}})_{V_{1}}
       -\sqrt{-1}\sum_{\alpha>1}\mu_{\alpha}\vp_{V_{\alpha}V_{1}i}+E,
\end{split}
\end{equation}
where $E$ denotes a term satisfying $|E|\leq C\lambda_{1}$. Combining (\ref{Bad term 3 equation 1}) with the Cauchy-Schwarz inequality, we compute
\begin{equation}\label{Bad term 3 equation 2}
\begin{split}
K_{3}
\leq {} & \frac{C}{\ve}\sum_{i\notin I}\frac{\ti{g}^{i\ov{i}}}{\lambda_{1}^{2}}
          \left|\sum_{q\notin I}\ov{\nu_{q}}(\ti{g}_{i\ov{q}})_{V_{1}}\right|^{2}
          +\frac{C}{\ve}\sum_{i}\ti{g}^{i\ov{i}} \\
        & +(1-\ve)\sum_{i\notin I}\frac{\ti{g}^{i\ov{i}}|\sqrt{2}\sum_{q\in I}\ov{\nu_{q}}(\ti{g}_{i\ov{q}})_{V_{1}}
          -\sqrt{-1}\sum_{\alpha>1}\mu_{\alpha}\vp_{V_{\alpha}V_{1}i}|^{2}}{\lambda_{1}^{2}}.
\end{split}
\end{equation}
For convenience, we write $I=\{1,2,\cdots,j\}$. Combining (\ref{Definition of e i}) and Lemma \ref{Partial second order estimate}, it is clear that
\begin{equation*}
\sum_{\alpha=2j+1}^{2n}\sum_{\beta=1}^{2n}|\vp_{\alpha\beta}| \leq C_{A}.
\end{equation*}
Since $V_{1}$ is the eigenvector of $\nabla^{2}\vp$ corresponding to $\lambda_{1}$, we have
\begin{equation*}
|V_{1}^{\alpha}| = \left|\frac{1}{\lambda_{1}}\sum_{\beta=1}^{2n}\vp_{\alpha\beta}V_{1}^{\beta}\right|
\leq \frac{C_{A}}{\lambda_{1}} ~\text{~for $\alpha=2j+1,\cdots,2n$}.
\end{equation*}
Recalling the definitions of $\nu_{q}$ (see (\ref{Expression of tilde e 1})) and $e_{i}$ (see (\ref{Definition of e i})), we obtain
\begin{equation}\label{Bad term 3 equation 5}
|\nu_{q}| \leq |V_{1}^{2q-1}|+|V_{1}^{2q}| \leq \frac{C_{A}}{\lambda_{1}} ~\text{~for $q\notin I$}.
\end{equation}

For the first term of (\ref{Bad term 3 equation 2}), by (\ref{Bad term 3 equation 5}), we compute
\begin{equation}\label{Bad term 3 equation 3}
\begin{split}
& \frac{C}{\ve}\sum_{i\notin I}\frac{\ti{g}^{i\ov{i}}}{\lambda_{1}^{2}}
\left|\sum_{q\notin I}\ov{\nu_{q}}(\ti{g}_{i\ov{q}})_{V_{1}}\right|^{2} \\
\leq {} & \frac{C_{A}}{\ve}\sum_{i\notin I}\sum_{q\notin I}\frac{\ti{g}^{i\ov{i}}|(\ti{g}_{i\ov{q}})_{V_{1}}|^{2}}{\lambda_{1}^{4}} \\
\leq {} & \frac{C_{A}}{\ve}\sum_{i\notin I}\sum_{q\notin I,q\neq i}\frac{\ti{g}^{i\ov{i}}|(\ti{g}_{i\ov{q}})_{V_{1}}|^{2}}{\lambda_{1}^{4}}
          +\frac{C_{A}}{\ve}\sum_{i\notin I}\frac{\ti{g}^{i\ov{i}}|(\ti{g}_{i\ov{i}})_{V_{1}}|^{2}}{\lambda_{1}^{4}}.
\end{split}
\end{equation}
Using (\ref{Expression of tilde g}), we obtain $\ti{g}_{q\ov{q}}\leq C\lambda_{1}$ for any $q$. Hence, if $\lambda_{1}\geq\frac{C_{A}}{\ve}$, then we have
\begin{equation*}
\frac{C_{A}}{\ve\lambda_{1}^{3}} \leq \ti{g}^{q\ov{q}},
\end{equation*}
which implies
\begin{equation}\label{Bad term 3 equation 6}
\frac{C_{A}}{\ve}\sum_{i\notin I}\sum_{q\notin I,q\neq i}\frac{\ti{g}^{i\ov{i}}|(\ti{g}_{i\ov{q}})_{V_{1}}|^{2}}{\lambda_{1}^{4}}
\leq \sum_{i\notin I}\sum_{q\notin I,q\neq i}\frac{\ti{g}^{i\ov{i}}\ti{g}^{q\ov{q}}|(\ti{g}_{i\ov{q}})_{V_{1}}|^{2}}{\lambda_{1}}.
\end{equation}
By (\ref{Expression of tilde g}) and (\ref{Commutation formula}), we see that
\begin{equation}\label{Bad term 3 equation 9}
\begin{split}
(\ti{g}_{i\ov{i}})_{k} = {} & \vp_{i\ov{i}k}+(S_{i\ov{i}}^{p}\vp_{p})_{k}+(S_{i\ov{i}}^{\ov{p}}\vp_{\ov{p}})_{k} \\
                       = {} & \vp_{k\ov{i}i}+E \\
                       = {} & (\ti{g}_{k\ov{i}})_{i}+E,
\end{split}
\end{equation}
where $E$ denotes a term satisfying $|E|\leq C\lambda_{1}$. Using (\ref{Definition of tilde e 1}) and (\ref{Expression of tilde e 1}), it is clear that
\begin{equation}\label{Bad term 3 equation 14}
V_{1} = \frac{1}{\sqrt{2}}(\ti{e}_{1}+\ov{\ti{e}}_{1}) = \frac{1}{\sqrt{2}}\sum_{k}(\nu_{k}e_{k}+\ov{\nu_{k}}\,\ov{e}_{k}).
\end{equation}
Combining $\lambda_{1}\geq\frac{C_{A}}{\ve}$, (\ref{Property 2 of h}) and (\ref{Second order estimate equation 4}), we have
\begin{equation}\label{Bad term 3 equation 15}
\frac{C_{A}}{\ve\lambda_{1}^{4}} \leq \frac{1}{15M_{R}^{2}} \leq 2h_{1}'
~\text{~and~}~ \frac{C_{A}}{\ve\lambda_{1}^{2}} \leq 1.
\end{equation}
From (\ref{Bad term 3 equation 9}), (\ref{Bad term 3 equation 14}) and (\ref{Bad term 3 equation 15}), it follows that
\begin{equation}\label{Bad term 3 equation 7}
\begin{split}
\frac{C_{A}}{\ve}\sum_{i\notin I}\frac{\ti{g}^{i\ov{i}}|(\ti{g}_{i\ov{i}})_{V_{1}}|^{2}}{\lambda_{1}^{4}}
  =  {} & \frac{C_{A}}{\ve\lambda_{1}^{4}}\sum_{i\notin I}\ti{g}^{i\ov{i}}
          \left|\sum_{k}\frac{\nu_{k}(\ti{g}_{i\ov{i}})_{k}+\ov{\nu_{k}}(\ti{g}_{i\ov{i}})_{\ov{k}}}{\sqrt{2}}\right|^{2} \\
\leq {} &  \frac{C_{A}}{\ve\lambda_{1}^{4}}\sum_{i,k}\ti{g}^{i\ov{i}}|(\ti{g}_{i\ov{i}})_{k}|^{2} \\
\leq {} & \frac{C_{A}}{\ve\lambda_{1}^{4}}\sum_{i,k}\ti{g}^{i\ov{i}}|(\ti{g}_{k\ov{i}})_{i}|^{2}
          +\frac{C_{A}}{\ve\lambda_{1}^{2}}\sum_{i}\ti{g}^{i\ov{i}} \\
\leq {} & 2h_{1}'\sum_{i,k,l}\ti{g}^{i\ov{i}}|(\ti{g}_{k\ov{l}})_{i}|^{2}+\sum_{i}\ti{g}^{i\ov{i}}.
\end{split}
\end{equation}
Substituting (\ref{Bad term 3 equation 6}) and (\ref{Bad term 3 equation 7}) into (\ref{Bad term 3 equation 3}), we obtain
\begin{equation}\label{Bad term 3 equation 8}
\begin{split}
        & \frac{C}{\ve}\sum_{i\notin I}\frac{\ti{g}^{i\ov{i}}}{\lambda_{1}^{2}}
          \left|\sum_{q\notin I}\ov{\nu_{q}}(\ti{g}_{i\ov{q}})_{V_{1}}\right|^{2} \\
\leq {} & \sum_{i\notin I}\sum_{q\notin I,q\neq i}\frac{\ti{g}^{i\ov{i}}\ti{g}^{q\ov{q}}|(\ti{g}_{i\ov{q}})_{V_{1}}|^{2}}{\lambda_{1}}
          +2h_{1}'\sum_{k,l}\ti{g}^{i\ov{i}}|(\ti{g}_{k\ov{l}})_{i}|^{2}+\sum_{i}\ti{g}^{i\ov{i}}.
\end{split}
\end{equation}

Next, we deal with the third term of (\ref{Bad term 3 equation 2}). For any $\gamma>0$, we have
\begin{equation}\label{Bad term 3 equation 10}
\begin{split}
& (1-\ve)\sum_{i\notin I}\frac{\ti{g}^{i\ov{i}}|\sqrt{2}\sum_{q\in I}\ov{\nu_{q}}(\ti{g}_{i\ov{q}})_{V_{1}}
  -\sqrt{-1}\sum_{\alpha>1}\mu_{\alpha}\vp_{V_{\alpha}V_{1}i}|^{2}}{\lambda_{1}^{2}} \\
\leq {} & (1-\ve)(1+\gamma)\sum_{i\notin I}\frac{2\ti{g}^{i\ov{i}}}{\lambda_{1}^{2}}
          \left|\sum_{q\in I}\ov{\nu_{q}}(\ti{g}_{i\ov{q}})_{V_{1}}\right|^{2} \\
        & +(1-\ve)\left(1+\frac{1}{\gamma}\right)\sum_{i\notin I}\frac{\ti{g}^{i\ov{i}}}{\lambda_{1}^{2}}
          \left|\sum_{\alpha>1}\mu_{\alpha}\vp_{V_{\alpha}V_{1}i}\right|^{2}.
\end{split}
\end{equation}
Using (\ref{Definition of tilde e 1}), (\ref{Expression of tilde e 1}) and the Cauchy-Schwarz inequality, we have
\begin{equation}\label{Bad term 3 equation 11}
\begin{split}
\sum_{i\notin I}\frac{2\ti{g}^{i\ov{i}}}{\lambda_{1}^{2}}
\left|\sum_{q\in I}\ov{\nu_{q}}(\ti{g}_{i\ov{q}})_{V_{1}}\right|^{2}
\leq {} & \sum_{i\notin I}\frac{2\ti{g}^{i\ov{i}}}{\lambda_{1}^{2}}
          \left(\sum_{q}|\nu_{q}|^{2}\ti{g}_{q\ov{q}}\right)\left(\sum_{q\in I}\ti{g}^{q\ov{q}}|(\ti{g}_{i\ov{q}})_{V_{1}}|^{2}\right) \\
  =  {} & \ti{g}(\ti{e}_{1},\ov{\ti{e}}_{1})\sum_{i\notin I}\sum_{q\in I}
          \frac{2\ti{g}^{i\ov{i}}\ti{g}^{q\ov{q}}|(\ti{g}_{i\ov{q}})_{V_{1}}|^{2}}{\lambda_{1}^{2}}
\end{split}
\end{equation}
and
\begin{equation}\label{Bad term 3 equation 12}
\begin{split}
          \sum_{i\notin I}\frac{\ti{g}^{i\ov{i}}}{\lambda_{1}^{2}}
          \left|\sum_{\alpha>1}\mu_{\alpha}\vp_{V_{\alpha}V_{1}i}\right|^{2}
\leq {} & \sum_{i\notin I}\frac{\ti{g}^{i\ov{i}}}{\lambda_{1}^{2}}\left(\sum_{\alpha>1}(\lambda_{1}-\lambda_{\alpha})\mu_{\alpha}^{2}\right)
          \left(\sum_{\alpha>1}\frac{|\vp_{V_{\alpha}V_{1}i}|^{2}}{\lambda_{1}-\lambda_{\alpha}}\right) \\
\leq {} & \left(\lambda_{1}-\sum_{\alpha>1}\lambda_{\alpha}\mu_{\alpha}^{2}\right)\sum_{i\notin I}\sum_{\alpha>1}
          \frac{\ti{g}^{i\ov{i}}|\vp_{V_{\alpha}V_{1}i}|^{2}}{\lambda_{1}^{2}(\lambda_{1}-\lambda_{\alpha})},
\end{split}
\end{equation}
where we used $\sum_{\alpha>1}\mu_{\alpha}^{2}=1$ (see (\ref{Expression of JV 1})) in the last inequality. For convenience, we denote $\ti{g}(\ti{e}_{1},\ov{\ti{e}}_{1})$ by $\ti{g}_{\ti{1}\ov{\ti{1}}}$. Substituting (\ref{Bad term 3 equation 11}) and (\ref{Bad term 3 equation 12}) into (\ref{Bad term 3 equation 10}), we have
\begin{equation}\label{Bad term 3 equation 13}
\begin{split}
&(1-\ve)\sum_{i\notin I}\frac{\ti{g}^{i\ov{i}}|\sqrt{2}\sum_{q\in I}\ov{\nu_{q}}(\ti{g}_{i\ov{q}})_{V_{1}}
  -\sqrt{-1}\sum_{\alpha>1}\mu_{\alpha}\vp_{V_{\alpha}V_{1}i}|^{2}}{\lambda_{1}^{2}} \\
\leq {} & (1-\ve)(1+\gamma)\ti{g}_{\ti{1}\ov{\ti{1}}}\sum_{i\notin I}\sum_{q\in I}
          \frac{2\ti{g}^{i\ov{i}}\ti{g}^{q\ov{q}}|(\ti{g}_{i\ov{q}})_{V_{1}}|^{2}}{\lambda_{1}^{2}} \\
        & +(1-\ve)\left(1+\frac{1}{\gamma}\right)\left(\lambda_{1}-\sum_{\alpha>1}\lambda_{\alpha}\mu_{\alpha}^{2}\right)
          \sum_{i\notin I}\sum_{\alpha>1}\frac{\ti{g}^{i\ov{i}}|\vp_{V_{\alpha}V_{1}i}|^{2}}{\lambda_{1}^{2}(\lambda_{1}-\lambda_{\alpha})}.
\end{split}
\end{equation}
Substituting (\ref{Bad term 3 equation 8}) and (\ref{Bad term 3 equation 13}) into (\ref{Bad term 3 equation 2}), it is clear that
\begin{equation}\label{Bad term 3 equation 4}
\begin{split}
K_{3}
\leq {} & (1-\ve)\left(1+\frac{1}{\gamma}\right)\left(\lambda_{1}-\sum_{\alpha>1}\lambda_{\alpha}\mu_{\alpha}^{2}\right)
          \sum_{i\notin I}\sum_{\alpha>1}\frac{\ti{g}^{i\ov{i}}|\vp_{V_{\alpha}V_{1}i}|^{2}}{\lambda_{1}^{2}(\lambda_{1}-\lambda_{\alpha})} \\
         & +(1-\ve)(1+\gamma)\ti{g}_{\ti{1}\ov{\ti{1}}}\sum_{i\notin I}\sum_{q\in I}
          \frac{2\ti{g}^{i\ov{i}}\ti{g}^{q\ov{q}}|(\ti{g}_{i\ov{q}})_{V_{1}}|^{2}}{\lambda_{1}^{2}}
          +\frac{C}{\ve}\sum_{i}\ti{g}^{i\ov{i}} \\
        & +\sum_{i\notin I}\sum_{q\notin I,q\neq i}\frac{\ti{g}^{i\ov{i}}\ti{g}^{q\ov{q}}|(\ti{g}_{i\ov{q}})_{V_{1}}|^{2}}{\lambda_{1}}
          +2h_{1}'\sum_{k,l}\ti{g}^{i\ov{i}}|(\ti{g}_{k\ov{l}})_{i}|^{2}.
\end{split}
\end{equation}

Next, we give the proof of Lemma \ref{Bad term 3}. We split up into two cases. The constant $\gamma>0$ will be different in each case.

\bigskip
\noindent

{\bf Case 1.} At $x_{0}$, we assume that
\begin{equation}\label{Case 1 equation 1}
\frac{1}{2}\left(\lambda_{1}+\sum_{\alpha>1}\lambda_{\alpha}\mu_{\alpha}^{2}\right)
\geq (1-\ve)\ti{g}_{\ti{1}\ov{\ti{1}}} > 0.
\end{equation}

\bigskip

Since $\sum_{\alpha>1}\mu_{\alpha}^{2}=1$ (see (\ref{Expression of JV 1})), it is clear that
\begin{equation*}
\lambda_{1}-\sum_{\alpha>1}\lambda_{\alpha}\mu_{\alpha}^{2}
= \sum_{\alpha>1}(\lambda_{1}-\lambda_{\alpha})\mu_{\alpha}^{2}
> 0.
\end{equation*}
Combining this with (\ref{Case 1 equation 1}), we have
\begin{equation}\label{Case 1 equation 2}
\gamma := \frac{\lambda_{1}-\sum_{\alpha>1}\lambda_{\alpha}\mu_{\alpha}^{2}}{\lambda_{1}+\sum_{\alpha>1}\lambda_{\alpha}\mu_{\alpha}^{2}} > 0.
\end{equation}
Substituting (\ref{Case 1 equation 1}) and (\ref{Case 1 equation 2}) into (\ref{Bad term 3 equation 4}), we compute
\begin{equation*}
\begin{split}
K_{3}
\leq {} & 2(1-\ve)\sum_{i\notin I}\sum_{\alpha>1}\frac{\ti{g}^{i\ov{i}}|\vp_{V_{\alpha}V_{1}i}|^{2}}
          {\lambda_{1}(\lambda_{1}-\lambda_{\alpha})}
          +2\sum_{i\notin I}\sum_{q\in I}
          \frac{\ti{g}^{i\ov{i}}\ti{g}^{q\ov{q}}|(\ti{g}_{i\ov{q}})_{V_{1}}|^{2}}{\lambda_{1}} \\
        & +\sum_{i\notin I}\sum_{q\notin I,q\neq i}\frac{\ti{g}^{i\ov{i}}\ti{g}^{q\ov{q}}|(\ti{g}_{i\ov{q}})_{V_{1}}|^{2}}{\lambda_{1}}
          +2h_{1}'\sum_{k,l}\ti{g}^{i\ov{i}}|(\ti{g}_{k\ov{l}})_{i}|^{2}
          +\frac{C}{\ve}\sum_{i}\ti{g}^{i\ov{i}} \\
\leq {} & 2\sum_{\alpha>1}\frac{\ti{g}^{i\ov{i}}|\vp_{V_{\alpha}V_{1}i}|^{2}}{\lambda_{1}(\lambda_{1}-\lambda_{\alpha})}
          +\sum_{p\neq q}\frac{\ti{g}^{p\ov{p}}\ti{g}^{q\ov{q}}|(\ti{g}_{p\ov{q}})_{V_{1}}|^{2}}{\lambda_{1}}
          +2h_{1}'\sum_{k,l}\ti{g}^{i\ov{i}}|(\ti{g}_{k\ov{l}})_{i}|^{2}
          +\frac{C}{\ve}\sum_{i}\ti{g}^{i\ov{i}},
\end{split}
\end{equation*}
which completes Case 1.

\bigskip
\noindent

{\bf Case 2.} At $x_{0}$, we assume that
\begin{equation}\label{Case 2 equation 1}
\frac{1}{2}\left(\lambda_{1}+\sum_{\alpha>1}\lambda_{\alpha}\mu_{\alpha}^{2}\right)
< (1-\ve)\ti{g}_{\ti{1}\ov{\ti{1}}}.
\end{equation}

\bigskip

Using (\ref{Expression of tilde g}), (\ref{Definition of tilde e 1}) and (\ref{Expression of JV 1}), we compute
\begin{equation}\label{Case 2 equation 2}
\begin{split}
  0  <    \ti{g}_{\ti{1}\ov{\ti{1}}}
  =  {} & \ti{g}(\ti{e}_{1},\ov{\ti{e}}_{1}) \\
  =  {} & 1+\de\dbar\vp(\ti{e}_{1},\ov{\ti{e}}_{1}) \\[1mm]
\leq {} & 1+(\nabla^{2}\vp)(\ti{e}_{1},\ov{\ti{e}}_{1})+C \\
\leq {} & 1+\frac{1}{2}(\vp_{V_{1}V_{1}}+\vp_{JV_{1}JV_{1}})+C \\
  =  {} & \frac{1}{2}\left(\lambda_{1}+\sum_{\alpha>1}\lambda_{\alpha}\mu_{\alpha}^{2}\right)+C.
\end{split}
\end{equation}
Combining (\ref{Case 2 equation 1}) and (\ref{Case 2 equation 2}), it is clear that
\begin{equation*}
\ti{g}_{\ti{1}\ov{\ti{1}}}
\leq \frac{1}{2}\left(\lambda_{1}+\sum_{\alpha>1}\lambda_{\alpha}\mu_{\alpha}^{2}\right)+C
\leq (1-\ve)\ti{g}_{\ti{1}\ov{\ti{1}}}+C,
\end{equation*}
which implies
\begin{equation}\label{Case 2 equation 3}
\ti{g}_{\ti{1}\ov{\ti{1}}} \leq \frac{C}{\ve}.
\end{equation}
Using (\ref{Case 2 equation 2}) again, we have
\begin{equation}\label{Case 2 equation 4}
\lambda_{1}-\sum_{\alpha>1}\lambda_{\alpha}\mu_{\alpha}^{2}
\leq 2\lambda_{1}+C
\leq 2(1+\ve^{2})\lambda_{1},
\end{equation}
as long as $\lambda_{1}\geq\frac{C}{\ve^{2}}$.

Now, we choose
\begin{equation*}
\gamma := \frac{1}{\ve^{2}}.
\end{equation*}
Combining (\ref{Case 2 equation 3}), (\ref{Case 2 equation 4}) and $\ve\in(0,\frac{1}{3})$, we have
\begin{equation}\label{Case 2 equation 5}
\begin{split}
        & (1-\ve)\left(1+\frac{1}{\gamma}\right)\left(\lambda_{1}-\sum_{\alpha>1}\lambda_{\alpha}\mu_{\alpha}^{2}\right)
          \sum_{i\notin I}\sum_{\alpha>1}\frac{\ti{g}^{i\ov{i}}|\vp_{V_{\alpha}V_{1}i}|^{2}}{\lambda_{1}^{2}(\lambda_{1}-\lambda_{\alpha})} \\
\leq {} & (1-\ve)(1+\ve^{2})(2+2\ve^{2})\lambda_{1}
          \sum_{i\notin I}\sum_{\alpha>1}\frac{\ti{g}^{i\ov{i}}|\vp_{V_{\alpha}V_{1}i}|^{2}}{\lambda_{1}^{2}(\lambda_{1}-\lambda_{\alpha})} \\
\leq {} & 2\sum_{i\notin I}\sum_{\alpha>1}\frac{\ti{g}^{i\ov{i}}|\vp_{V_{\alpha}V_{1}i}|^{2}}{\lambda_{1}(\lambda_{1}-\lambda_{\alpha})}.
\end{split}
\end{equation}
From (\ref{Case 2 equation 1}) and (\ref{Case 2 equation 3}), it follows that
\begin{equation}\label{Case 2 equation 6}
\begin{split}
        & (1-\ve)(1+\gamma)\ti{g}_{\ti{1}\ov{\ti{1}}}\sum_{i\notin I}\sum_{q\in I}
          \frac{2\ti{g}^{i\ov{i}}\ti{g}^{q\ov{q}}|(\ti{g}_{i\ov{q}})_{V_{1}}|^{2}}{\lambda_{1}^{2}} \\
\leq {} & (1-\ve)\left(1+\frac{1}{\ve^{2}}\right)\frac{C}{\ve}\sum_{i\notin I}\sum_{q\in I}
          \frac{2\ti{g}^{i\ov{i}}\ti{g}^{q\ov{q}}|(\ti{g}_{i\ov{q}})_{V_{1}}|^{2}}{\lambda_{1}^{2}} \\
\leq {} & \frac{C}{\ve^{3}}\sum_{i\notin I}\sum_{q\in I}
          \frac{2\ti{g}^{i\ov{i}}\ti{g}^{q\ov{q}}|(\ti{g}_{i\ov{q}})_{V_{1}}|^{2}}{\lambda_{1}^{2}} \\
\leq {} & 2\sum_{i\notin I}\sum_{q\in I}
          \frac{\ti{g}^{i\ov{i}}\ti{g}^{q\ov{q}}|(\ti{g}_{i\ov{q}})_{V_{1}}|^{2}}{\lambda_{1}},
\end{split}
\end{equation}
as long as $\lambda_{1}\geq\frac{C}{\ve^{3}}$. Substituting (\ref{Case 2 equation 5}) and (\ref{Case 2 equation 6}) into (\ref{Bad term 3 equation 4}), we get
\begin{equation*}
\begin{split}
K_{3}
\leq {} & 2\sum_{i\notin I}\sum_{\alpha>1}\frac{\ti{g}^{i\ov{i}}|\vp_{V_{\alpha}V_{1}i}|^{2}}
          {\lambda_{1}(\lambda_{1}-\lambda_{\alpha})}
          +2\sum_{i\notin I}\sum_{q\in I}
          \frac{\ti{g}^{i\ov{i}}\ti{g}^{q\ov{q}}|(\ti{g}_{i\ov{q}})_{V_{1}}|^{2}}{\lambda_{1}} \\
        & +\sum_{i\notin I}\sum_{q\notin I,q\neq i}\frac{\ti{g}^{i\ov{i}}\ti{g}^{q\ov{q}}|(\ti{g}_{i\ov{q}})_{V_{1}}|^{2}}{\lambda_{1}}
          +2h_{1}'\sum_{k,l}\ti{g}^{i\ov{i}}|(\ti{g}_{k\ov{l}})_{i}|^{2}
          +\frac{C}{\ve}\sum_{i}\ti{g}^{i\ov{i}} \\
\leq {} & 2\sum_{\alpha>1}\frac{\ti{g}^{i\ov{i}}|\vp_{V_{\alpha}V_{1}i}|^{2}}{\lambda_{1}(\lambda_{1}-\lambda_{\alpha})}
          +\sum_{p\neq q}\frac{\ti{g}^{p\ov{p}}\ti{g}^{q\ov{q}}|(\ti{g}_{p\ov{q}})_{V_{1}}|^{2}}{\lambda_{1}}
          +2h_{1}'\sum_{k,l}\ti{g}^{i\ov{i}}|(\ti{g}_{k\ov{l}})_{i}|^{2}
          +\frac{C}{\ve}\sum_{i}\ti{g}^{i\ov{i}},
\end{split}
\end{equation*}
which completes Case 2.
\end{proof}

Combining Lemma \ref{Bad term 1 and 2} and \ref{Bad term 3}, we obtain an upper bound of the "bad" third order term $K$:
\begin{equation}\label{Second order estimate equation 1}
\begin{split}
K  = {} & K_{1}+K_{2}+K_{3} \\
\leq {} & 2\sum_{\alpha>1}\frac{\ti{g}^{i\ov{i}}|\vp_{V_{\alpha}V_{1}i}|^{2}}{\lambda_{1}(\lambda_{1}-\lambda_{\alpha})}
          +\sum_{p\neq q}\frac{\ti{g}^{p\ov{p}}\ti{g}^{q\ov{q}}|(\ti{g}_{p\ov{q}})_{V_{1}}|^{2}}{\lambda_{1}}
          +2h_{1}'\sum_{k,l}\ti{g}^{i\ov{i}}|(\ti{g}_{k\ov{l}})_{i}|^{2} \\
        & +3(h_{1}')^{2}\ti{g}^{i\ov{i}}|(|\ti{\omega}|_{g}^{2})_{i}|^{2}
          +3(h_{2}')^{2}\ti{g}^{i\ov{i}}|(|\de\vp|_{g}^{2})_{i}|^{2} \\[1mm]
        & +6\ve A^{2}e^{-2A\vp}\ti{g}^{i\ov{i}}|\vp_{i}|^{2}+\frac{C}{\ve}\sum_{i}\ti{g}^{i\ov{i}}.
\end{split}
\end{equation}

Now we are in a position to prove Proposition \ref{Second order estimate}.

\begin{proof}[Proof of Proposition \ref{Second order estimate}]
Combining Lemma \ref{Main inequality}, (\ref{Bad term K}), (\ref{Second order estimate equation 1}) and (\ref{Property 1 of h}), it is clear that
\begin{equation*}
\begin{split}
0 \geq {} & h_{2}'\sum_{k}\ti{g}^{i\ov{i}}(|\vp_{ik}|^{2}+|\vp_{i\ov{k}}|^{2})
            +(A^{2}e^{-A\vp}-6\ve A^{2}e^{-2A\vp})\ti{g}^{i\ov{i}}|\vp_{i}|^{2} \\
          & +\left(Ae^{-A\vp}-\frac{C_{0}}{\ve}\right)\sum_{i}\ti{g}^{i\ov{i}}-Ane^{-A\vp},
\end{split}
\end{equation*}
where $C_{0}$ is a uniform constant. We choose
\begin{equation*}
A = 6C_{0}+1 \text{~and~} \ve = \frac{e^{A\vp(x_{0})}}{6}.
\end{equation*}
Recalling $\sup_{M}\vp=0$, we see that
\begin{equation*}
A^{2}e^{-A\vp}-\ve A^{2}e^{-2A\vp} \geq \frac{1}{6}
\text{~and~} Ae^{-A\vp}-\frac{C_{0}}{\ve} \geq 1.
\end{equation*}
It then follows that
\begin{equation}\label{Second order estimate equation 2}
h_{2}'\sum_{k}\ti{g}^{i\ov{i}}(|\vp_{ik}|^{2}+|\vp_{i\ov{k}}|^{2})+\sum_{i}\ti{g}^{i\ov{i}} \leq C.
\end{equation}
From $\sum_{i}\ti{g}^{i\ov{i}}\leq C$ and $\frac{\det\ti{g}}{\det g}\leq C$, we have $\ti{g}_{i\ov{i}}\leq C$ for each $i$. Combining this with (\ref{Property 2 of h}) and (\ref{Second order estimate equation 2}), we obtain $\lambda_{1}(x_{0})\leq C$. By (\ref{Second order estimate equation 4}), it is clear that
\begin{equation*}
\sup_{M}|\nabla^{2}\vp|_{g}+1 = M_{R} \leq C,
\end{equation*}
as desired.
\end{proof}

\section{Proofs of Theorem \ref{Main Theorem}, \ref{Application DMAE}, \ref{Application SMAE}, \ref{Application SMAE measure} and \ref{Application Dirichlet Problem}}
In this section, we prove Theorem \ref{Main Theorem}, \ref{Application DMAE}, \ref{Application SMAE}, \ref{Application SMAE measure} and \ref{Application Dirichlet Problem}.

\begin{proof}[Proof of Theorem \ref{Main Theorem}]
Theorem \ref{Main Theorem} is an immediate consequence of Proposition \ref{Zero order estimate}, \ref{First order estimate} and \ref{Second order estimate}.
\end{proof}

\begin{proof}[Proof of Theorem \ref{Application DMAE}]
By the assumptions of Theorem \ref{Application DMAE}, there exists a sequence of positive smooth function $f_{i}$ on $M$ such that
\begin{enumerate}
\item $\lim_{i\rightarrow\infty}\|f_{i}-f\|_{C^{0}}=0$;
\item for a uniform constant $C$,
\begin{equation*}
\sup_{M}f_{i} \leq C, \,\, \sup_{M}|\de(f_{i}^{\frac{1}{n}})|_{g}\leq C, \,\, \nabla^{2}(f_{i}^{\frac{1}{n}}) \geq -Cg.
\end{equation*}
\end{enumerate}
Using \cite[Theorem 1.1]{CTW16}, there exists a pair $(\vp_{i},b_{i})$ where $\vp_{i}\in C^{\infty}(M)$ and $b_{i}\in\mathbb{R}$, such that
\begin{equation}\label{Application DMAE equation 1}
\begin{split}
(\omega+\ddbar\vp_{i})^{n} = {} & f_{i}e^{b_{i}} \omega^{n}, \\
\quad \omega+\ddbar\vp_{i} > 0, \quad {} & \sup_{M}\vp_{i} = 0.
\end{split}
\end{equation}
We need to prove $|b_{i}|\leq C$. For the upper bound of $b_{i}$, by the arithmetic-geometric mean inequality, we obtain
\begin{equation}\label{Application DMAE equation 2}
\left(\frac{\det\ti{g}}{\det g}\right)^{\frac{1}{n}}
\leq \frac{1}{n}\left(n+\frac{n\ddbar\vp_{i}\wedge\omega^{n-1}}{\omega^{n}}\right)
\leq 1+\frac{(dJd\vp_{i})\wedge\omega^{n-1}}{2\omega^{n}}.
\end{equation}
Combining (\ref{Application DMAE equation 1}), (\ref{Application DMAE equation 2}) and the Stokes' formula, we compute
\begin{equation}\label{Application DMAE equation 3}
\begin{split}
\int_{M}f_{i}^{\frac{1}{n}}e^{\frac{b_{i}}{n}}\omega^{n}
\leq {} & \int_{M}\left(\frac{\det\ti{g}}{\det g}\right)^{\frac{1}{n}}\omega^{n} \\
\leq {} & \textrm{Vol}(M,\omega)+\frac{1}{2}\int_{M}(dJd\vp_{i})\wedge\omega^{n-1} \\
\leq {} & \textrm{Vol}(M,\omega)+\frac{1}{2}\int_{M}\vp_{i}(dJd\omega^{n-1}) \\
\leq {} & \textrm{Vol}(M,\omega)+C\int_{M}|\vp_{i}|\omega^{n}.
\end{split}
\end{equation}
Since $\lim_{i\rightarrow\infty}\|f_{i}-f\|_{C^{0}}=0$ and $f\not\equiv0$, we have
\begin{equation}\label{Application DMAE equation 4}
\frac{1}{2}\int_{M}f^{\frac{1}{n}}\omega^{n} \leq \int_{M}f_{i}^{\frac{1}{n}}\omega^{n} ~\text{~for sufficiently large $i$}.
\end{equation}
Using (\ref{Application DMAE equation 3}), (\ref{Application DMAE equation 4}) and Proposition \ref{L1 estimate}, it is clear that
\begin{equation}\label{Upper bound of b i}
e^{b_{i}} \leq \left(\frac{C}{\int_{M}f^{\frac{1}{n}}\omega^{n}}\right)^{n}.
\end{equation}
Next, we will prove $b_{i}\geq-C$. Let $x_{0}$ be the minimum point of $\vp_{i}$. By the maximum principle, we have
\begin{equation*}
e^{b_{i}}f_{i}(x_{0}) = \frac{(\omega+\ddbar\vp_{i})^{n}}{\omega^{n}}(x_{0}) \geq 1,
\end{equation*}
which implies
\begin{equation}\label{Lower bound of b i}
e^{b_{i}} \geq \frac{1}{\sup_{M}f_{i}} \geq \frac{1}{C}.
\end{equation}
Combining (\ref{Upper bound of b i}), (\ref{Lower bound of b i}) and Theorem \ref{Main Theorem}, we obtain
\begin{equation*}
|b_{i}|+\sup_{M}|\vp_{i}|+\sup_{M}|\de\vp_{i}|_{g}+\sup_{M}|\nabla^{2}\vp_{i}|_{g} \leq C.
\end{equation*}
After passing to a subsequence, we show the existence of $C^{1,1}$ solution to (\ref{DMAE}).
\end{proof}

\begin{proof}[Proof of Theorem \ref{Application SMAE}]
By (\ref{Application SMAE equation 4}), it is clear that $s\not\equiv 0$. For any $i\geq1$, we define
\begin{equation*}
f_{i} = (|s|_{h}^{2}+i^{-1})^{\frac{1}{2}}.
\end{equation*}
Thanks to Theorem \ref{Application DMAE}, it suffices to verify that
\begin{equation}\label{Application SMAE equation 1}
\sup_{M}|\de f_{i}|_{g} \leq C \text{~and~}
\nabla^{2}f_{i} \geq -Cg,
\end{equation}
for a constant $C$ which is independent of $i$. For any point $x_{0}\in M$, there exists a local section $s_{0}$ in a neighbourhood of $x_{0}$ such that $|s_{0}|_{h}^{2}\equiv1$. We write
\begin{equation*}
s=(s_{R}+\sqrt{-1}s_{I})s_{0},
\end{equation*}
where $s_{R}$ and $s_{I}$ are local functions near $x_{0}$. It then follows that
\begin{equation*}
|s|_{h}^{2} = s_{R}^{2}+s_{I}^{2}
\text{~and~}
f_{i} = (s_{R}^{2}+s_{I}^{2}+i^{-1})^{\frac{1}{2}}.
\end{equation*}
For any $g$-unit vector field $V$ near $x_{0}$, we compute
\begin{equation}\label{Application SMAE equation 5}
V(f_{i}) = \frac{s_{R}V(s_{R})+s_{I}V(s_{I})}{(s_{R}^{2}+s_{I}^{2}+i^{-1})^{\frac{1}{2}}},
\end{equation}
which implies
\begin{equation}\label{Application SMAE equation 2}
|V(f_{i})| \leq C.
\end{equation}
Applying $V$ to (\ref{Application SMAE equation 5}), we obtain
\begin{equation}\label{Application SMAE equation 3}
\begin{split}
VV(f_{i}) = {} & \frac{s_{R}VV(s_{R})+s_{I}VV(s_{I})+(V(s_{R}))^{2}+(V(s_{I}))^{2}}{(s_{R}^{2}+s_{I}^{2}+i^{-1})^{\frac{1}{2}}} \\
                 & -\frac{(s_{R}V(s_{R})+s_{I}V(s_{I}))^{2}}{(s_{R}^{2}+s_{I}^{2}+i^{-1})^{\frac{3}{2}}} \\
         \geq {} & \frac{s_{R}VV(s_{R})+s_{I}VV(s_{I})}{(s_{R}^{2}+s_{I}^{2}+i^{-1})^{\frac{1}{2}}}
                   +\frac{(s_{R}V(s_{I})+s_{I}V(s_{R}))^{2}}{(s_{R}^{2}+s_{I}^{2}+i^{-1})^{\frac{3}{2}}} \\[2mm]
         \geq {} & -C.
\end{split}
\end{equation}
Since $x_{0}$ and $V$ are arbitrary, (\ref{Application SMAE equation 1}) follows from (\ref{Application SMAE equation 2}) and (\ref{Application SMAE equation 3}).
\end{proof}

We will prove Theorem \ref{Application SMAE measure} by means of blow-up construction. For the reader's convenience, let us recall its definition first. Let $\ti{M}$ be the blow-up of $M$ at $p$ and $\pi:\ti{M}\rightarrow M$ be the projection map. We denote the exceptional divisor by $E$ (i.e., $E=\pi^{-1}(p)$). We fix a coordinate chart $(U;\{z^{i}\}_{i=1}^{n})$ centered at $p$, which we identify via $\{z^{i}\}_{i=1}^{n}$ with the unit ball $B_{1}\subset\mathbb{C}^{n}$. By the exposition in \cite{GH78}, we identify $\pi^{-1}(B_{1})$ with $\ti{U}$ given by
\begin{equation*}
\ti{U} = \{ (z,l)\in B_{1}\times\mathbb{CP}^{n-1}~|~z^{i}l^{j}=z^{j}l^{i} \},
\end{equation*}
where $l=[l^{1},\cdots,l^{n}]\in\mathbb{CP}^{n-1}$. We set
\begin{equation*}
\ti{U}_{i} = \{ (z,l)\in B_{1}\times\mathbb{CP}^{n-1}~|~l^{i}\neq0 \}.
\end{equation*}
In $\ti{U}_{i}$, we have local coordinates $\{w_{i}^{j}\}_{j=1}^{n}$:
\begin{equation*}
w_{i}^{i} = z^{i} ~\text{~and~}~ w_{i}^{j} = \frac{l^{j}}{l^{i}} \text{~for $j\neq i$}.
\end{equation*}
Hence, $\{(\ti{U}_{i},\{w_{i}^{j}\}_{j=1}^{n})\}$ is a family of coordinate charts satisfying
\begin{equation*}
\ti{U} = \bigcup_{i=1}^{n}\ti{U}_{i}.
\end{equation*}
The projection map $\pi:\ti{M}\rightarrow M$ is given in $\ti{U}_{i}$ by
\begin{equation}\label{Blow-up equation 1}
(w_{i}^{1},\cdots,w_{i}^{n})\rightarrow(w_{i}^{i}w_{i}^{1},\cdots,w_{i}^{i},\cdots,w_{i}^{i}w_{i}^{n})
\end{equation}
and $E\cap\ti{U}_{i}$ is given by
\begin{equation*}
E\cap\ti{U}_{i} = \{ (z,l)\in B_{1}\times\mathbb{CP}^{n-1}~|~w_{i}^{i}=z^{i}=0 \}.
\end{equation*}
The line bundle $[E]$ over $\ti{U}$ has transition functions
\begin{equation*}
t_{ij} = \frac{z^{i}}{z^{j}} \text{~on $\ti{U}_{i}\cap\ti{U}_{j}$}.
\end{equation*}
Let $s$ be the global section of $[E]$ over $\ti{M}$ by setting
\begin{equation*}
s =
\left\{ \begin{array}{ll}
z^{i} ~\text{~on $\ti{U}_{i}$}, \\[1mm]
1 ~\text{~\,\,on $\ti{M}\setminus\pi^{-1}(B_{\frac{1}{2}})$}.
\end{array}\right.
\end{equation*}
It follows that $\{s=0\}=E$. We construct a Hermitian metric $h$ on $[E]$ as follows. Let $h_{1}$ be the Hermitian metric over $\ti{U}$ defined by
\begin{equation*}
h_{1} = \frac{\sum_{j=1}^{n}|l^{j}|^{2}}{|l^{i}|^{2}} ~\text{~on $\ti{U}_{i}$},
\end{equation*}
and let $h_{2}$ be the Hermitian metric over $\ti{M}\setminus E$ such that $|s|_{h}^{2}=1$. Then we define
\begin{equation*}
h = \rho_{1}h_{1}+\rho_{2}h_{2},
\end{equation*}
where $\{\rho_{1},\rho_{2}\}$ is a partition of unity for the cover $\{\pi^{-1}(B_{1}),\ti{M}\setminus\pi^{-1}(B_{\frac{1}{2}})\}$ of $\ti{M}$. It follows that
\begin{equation*}
h = h_{1} \text{~on $\pi^{-1}(B_{\frac{1}{2}})$}.
\end{equation*}
If $\pi(\ti{x})=(z^{1},\cdots,z^{n})\in B_{\frac{1}{2}}$, then we have
\begin{equation}\label{Blow-up equation 2}
|s|_{h}^{2}(\ti{x}) = \sum_{i=1}^{n}|z^{i}|^{2} =: |z|^{2}.
\end{equation}
On $\pi^{-1}(B_{\frac{1}{2}}\setminus\{0\})$, the curvature $R(h)$ of the Hermitian metric $h$ is given by
\begin{equation*}
R(h) = -\frac{\sqrt{-1}}{2\pi}\de\dbar\log\left(\sum_{i=1}^{n}|z^{i}|^{2}\right)
     = -\frac{\sqrt{-1}}{2\pi}\de\dbar\log|z|^{2}.
\end{equation*}
When $\ve$ is sufficient small,
\begin{equation}\label{Blow-up equation 3}
\ti{\omega} = \pi^{*}\omega-\ve R(h)
\end{equation}
is a K\"{a}hler form on $\ti{M}$ (see \cite[p.178]{GH78}). Moreover, the function $\frac{(\pi^{*}\omega)^{n}}{\ti{\omega}^{n}}$ has analytic zeros of the form $|s|_{h}^{2n-2}$. More precisely, we have the following lemma.

\begin{lemma}\label{Blow-up lemma}
There exists a smooth function $\ti{F}$ on $\ti{M}$ such that
\begin{equation}\label{Blow-up lemma equation 1}
(\pi^{*}\omega)^{n} = |s|_{h}^{2n-2}e^{\ti{F}}\ti{\omega}^{n}.
\end{equation}
\end{lemma}

\begin{proof}
By the definition of blow-up construction, it is clear that
\begin{equation*}
\frac{(\pi^{*}\omega)^{n}}{\ti{\omega}^{n}}\neq0, \,\, |s|_{h}^{2}\neq0 ~\text{~on $\ti{M}\setminus E$}.
\end{equation*}
To prove Lemma \ref{Blow-up lemma}, it suffices to prove (\ref{Blow-up lemma equation 1}) near $E$. By the definition of $\ti{U}_{i}$, we have
\begin{equation*}
E \subset \bigcup_{i=1}^{n}\left(\ti{U}_{i}\cap\pi^{-1}(B_{\frac{1}{2}})\right).
\end{equation*}
Hence, our aim is to verify
\begin{equation}\label{Blow-up lemma equation 2}
(\pi^{*}\omega)^{n} = |s|_{h}^{2n-2}e^{\ti{F}}\ti{\omega}^{n} \text{~on $\ti{U}_{i}\cap\pi^{-1}(B_{\frac{1}{2}})$},
\text{~for each $i=1,2,\cdots,n$}.
\end{equation}
Without loss of generality, we only prove (\ref{Blow-up lemma equation 2}) when $i=1$. We use $\omega_{\textrm{Eucl}}$ and $\ti{\omega}_{\textrm{Eucl}}$ to denote the Euclidean metrics on $B_{1}$ and $\ti{U}_{1}$, i.e.,
\begin{equation*}
\omega_{\textrm{Eucl}} = \sqrt{-1}\sum_{j=1}^{n}dz^{j}\wedge d\ov{z^{j}}
\text{~and~}
\ti{\omega}_{\textrm{Eucl}} = \sqrt{-1}\sum_{j=1}^{n}dw_{1}^{j}\wedge d\ov{w_{1}^{j}}.
\end{equation*}
Using (\ref{Blow-up equation 1}), we compute
\begin{equation*}
\begin{split}
\pi^{*}\omega_{\textrm{Eucl}}
= {} & \pi^{*}\left(\sqrt{-1}\sum_{j=1}^{n}dz^{j}\wedge d\ov{z^{j}}\right) \\
= {} & \sqrt{-1}dw_{1}^{1}\wedge d\ov{w_{1}^{1}}+\sqrt{-1}\sum_{j=2}^{n}d(w_{1}^{1}w_{1}^{j})\wedge d\ov{(w_{1}^{1}w_{1}^{j})} \\
= {} & \sqrt{-1}\left(1+\sum_{j=2}^{n}|w_{1}^{j}|^{2}\right)dw_{1}^{1}\wedge d\ov{w_{1}^{1}}
       +\sqrt{-1}\sum_{j=2}^{n}w_{1}^{j}\ov{w_{1}^{1}}dw_{1}^{1}\wedge d\ov{w_{1}^{j}} \\
     & +\sqrt{-1}\sum_{j=2}^{n}w_{1}^{1}\ov{w_{1}^{j}}dw_{1}^{j}\wedge d\ov{w_{1}^{1}}
       +\sqrt{-1}\sum_{j=2}^{n}|w_{1}^{1}|^{2}dw_{1}^{j}\wedge d\ov{w_{1}^{j}}.
\end{split}
\end{equation*}
By direct calculation, we obtain
\begin{equation}\label{Blow-up lemma equation 3}
\begin{split}
(\pi^{*}\omega_{\textrm{Eucl}})^{n}
= {} & |w_{1}^{1}|^{2n-2}\left(1+\sum_{k=2}^{n}|w_{1}^{k}|^{2}\right)\left(\sqrt{-1}\sum_{j=1}^{n}dw_{1}^{j}\wedge d\ov{w_{1}^{j}}\right)^{n} \\
     & -|w_{1}^{1}|^{2n-2}\left(\sum_{k=2}^{n}|w_{1}^{k}|^{2}\right)\left(\sqrt{-1}\sum_{j=1}^{n}dw_{1}^{j}\wedge d\ov{w_{1}^{j}}\right)^{n} \\
= {} & |w_{1}^{1}|^{2n-2}\ti{\omega}_{\textrm{Eucl}}^{n}.
\end{split}
\end{equation}
By (\ref{Blow-up equation 1}) and (\ref{Blow-up equation 2}), it is clear that
\begin{equation}\label{Blow-up lemma equation 4}
|s|_{h}^{2}(w_{1}^{1},\cdots,w_{1}^{n})
= |w_{1}^{1}|^{2}\left(1+\sum_{j=2}^{n}|w_{1}^{j}|^{2}\right).
\end{equation}
Combining (\ref{Blow-up lemma equation 3}) and (\ref{Blow-up lemma equation 4}), we have
\begin{equation*}
(\pi^{*}\omega_{\textrm{Eucl}})^{n}
= |s|_{h}^{2n-2}\left(1+\sum_{j=2}^{n}|w_{1}^{j}|^{2}\right)^{-n+1}\ti{\omega}_{\textrm{Eucl}}^{n}.
\end{equation*}
It follows that
\begin{equation*}
\begin{split}
(\pi^{*}\omega)^{n}
= {} & \pi^{*}\left(\frac{\omega^{n}}{\omega_{\textrm{Eucl}}^{n}}\right)(\pi^{*}\omega_{\textrm{Eucl}})^{n} \\
= {} & \pi^{*}\left(\frac{\omega^{n}}{\omega_{\textrm{Eucl}}^{n}}\right)|s|_{h}^{2n-2}
       \left(1+\sum_{j=2}^{n}|w_{1}^{j}|^{2}\right)^{-n+1}\ti{\omega}_{\textrm{Eucl}}^{n} \\
= {} & |s|_{h}^{2n-2}\pi^{*}\left(\frac{\omega^{n}}{\omega_{\textrm{Eucl}}^{n}}\right)
       \left(1+\sum_{j=2}^{n}|w_{1}^{j}|^{2}\right)^{-n+1}\left(\frac{\ti{\omega}_{\textrm{Eucl}}^{n}}{\ti{\omega}^{n}}\right)\ti{\omega}^{n},
\end{split}
\end{equation*}
which implies (\ref{Blow-up lemma equation 2}), as desired.
\end{proof}

Now we are in a position to prove Theorem \ref{Application SMAE measure}.

\begin{proof}[Proof of Theorem \ref{Application SMAE measure}]
For convenience, we use the same notations as above. To prove Theorem \ref{Application SMAE measure}, we follow the approach of \cite{PS14}. By (\ref{Blow-up equation 3}), when $\ve_{0}$ is sufficiently small,
\begin{equation}\label{Application SMAE measure equation 1}
\ti{\omega} = \pi^{*}\omega+\ve_{0}\ddbar\left(\rho_{1}\log|z|^{2}+\rho_{2}\right)
\end{equation}
can be extended to a smooth K\"{a}hler form on $\ti{M}$. We still denote it by $\ti{\omega}$. By Lemma \ref{Blow-up lemma}, there exists a smooth function $\ti{F}$ on $\ti{M}$ such that
\begin{equation}\label{Application SMAE measure equation 2}
(\pi^{*}\omega)^{n} = |s|_{h}^{2n-2}e^{\ti{F}}\ti{\omega}^{n}.
\end{equation}
Thanks to Theorem \ref{Application SMAE}, there exists a pair $(\ti{\vp},b)$ where $\ti{\vp}\in C^{1,1}(\ti{M})$ and $b\in\mathbb{R}$, such that
\begin{equation*}
(\ti{\omega}+\ddbar\ti{\vp})^{n} = |s|_{h}^{2n-2}e^{F+\ti{F}+b}\ti{\omega}^{n}.
\end{equation*}
Combining this with (\ref{Application SMAE measure equation 2}), we see that
\begin{equation*}
(\ti{\omega}+\ddbar\ti{\vp})^{n} = e^{F+b}(\pi^{*}\omega)^{n}.
\end{equation*}
Restricting this to $\ti{M}\setminus E$ and using (\ref{Application SMAE measure equation 1}), it is clear that
\begin{equation*}
\left(\pi^{*}\omega+\ddbar\big(\ve_{0}(\rho_{1}\log|z|^{2}+\rho_{2})+\ti{\vp}\big)\right)^{n} = e^{F+b}(\pi^{*}\omega)^{n}
\text{~on $M\setminus\{p\}$.}
\end{equation*}
Defining $\vp=\ve_{0}\left(\rho_{1}\log|z|^{2}+\rho_{2}\right)+\ti{\vp}$, we obtain
\begin{equation*}
(\omega+\ddbar\vp)^{n} = (\ve_{0}\delta_{p}+e^{F+b})\omega^{n}.
\end{equation*}
Since $\int_{M}e^{F}\omega^{n}=\int_{M}\omega^{n}=1$, we have $e^{b}=1-\ve_{0}$. Then $\vp$ is the desired solution.
\end{proof}

\begin{proof}[Proof of Theorem \ref{Application Dirichlet Problem}]
It suffices to establish the boundary estimate. The zero and first order estimate were proved in \cite[Section 3]{Blocki12}. When $\de M$ is pseudoconcave, combining \cite[Lemma 7.16]{Boucksom12} and the argument of \cite[Theorem 3.2']{Blocki12}, we obtain the second order estimate. When $\de M$ is Levi-flat, the second order estimate was proved in \cite[Theorem 3.2']{Blocki12}.
\end{proof}

\section{Proof of Theorem \ref{C11 regularity of geodesics}}
In this section, we give the proof of Theorem \ref{C11 regularity of geodesics}. First, we generalize Proposition \ref{Second order estimate} in a slightly more setting. Let $(M,\omega,J)$ be a compact almost Hermitian manifold of real dimension $2n$, with nonempty smooth boundary. Let $\{e_{i}\}_{i=1}^{n}$ be a local frame for $T_{\mathbb{C}}^{(1,0)}M$. By (\ref{Expression of tilde g}), the complex Monge-Amp\`{e}re equation (\ref{CMAE}) can be expressed as
\begin{equation*}
\left\{ \begin{array}{ll}
  \det(g_{i\ov{j}}+\vp_{i\ov{j}}+S_{i\ov{j}}^{p}\vp_{p}+S_{i\ov{j}}^{\ov{p}}\vp_{\ov{p}}) = f\det(g_{i\ov{j}}), \\[2mm]
  g_{i\ov{j}}+\vp_{i\ov{j}}+S_{i\ov{j}}^{p}\vp_{p}+S_{i\ov{j}}^{\ov{p}}\vp_{\ov{p}}>0,
\end{array}\right.
\end{equation*}
where $\vp_{i\ov{j}}=(\nabla^{2}\vp)(e_{i},\ov{e}_{j})$, $\nabla$ is the Levi-Civita connection of $g$ and $S$ is a tensor field defined by (\ref{Definition of S}).

Motivated by this, we introduce the following complex Monge-Amp\`{e}re type equation
\begin{equation}\label{Complex Monge-Ampere type equation}
\left\{ \begin{array}{ll}
  \det(g_{i\ov{j}}+\vp_{i\ov{j}}+T_{i\ov{j}}^{p}\vp_{p}+T_{i\ov{j}}^{\ov{p}}\vp_{\ov{p}}) = f\det(g_{i\ov{j}}), \\[2mm]
  g_{i\ov{j}}+\vp_{i\ov{j}}+T_{i\ov{j}}^{p}\vp_{p}+T_{i\ov{j}}^{\ov{p}}\vp_{\ov{p}}>0,
\end{array}\right.
\end{equation}
where $T$ is a tensor field satisfying
\begin{equation}\label{Property of T}
\ov{T_{i\ov{j}}^{p}} = T_{j\ov{i}}^{\ov{p}} ~\text{~and~}~
\ov{T_{i\ov{j}}^{\ov{p}}} = T_{j\ov{i}}^{p}.
\end{equation}
The condition (\ref{Property of T}) can be regarded as an analog of (\ref{Property of S}).

\begin{proposition}\label{C11 interior estimate}
Let $\vp$ be a smooth solution of (\ref{Complex Monge-Ampere type equation}). Then there exists a constant $C$ depending only on $\sup_{M}|\vp|$, $\sup_{M}|\de\vp|_{g}$, $\sup_{\de M}|\nabla^{2}\vp|_{g}$, $(M,\omega,J)$, $\sup_{M}f$, $\sup_{M}|\de(f^{\frac{1}{n}})|_{g}$ and lower bound of $\nabla^{2}(f^{\frac{1}{n}})$ such that
\begin{equation*}
\sup_{M}|\nabla^{2}\vp|_{g} \leq C.
\end{equation*}
\end{proposition}

\begin{proof}
We omit the proof since it is almost identical to that of Proposition \ref{Second order estimate} (just replacing $S$ by $T$).
\end{proof}

Now we are in a position to prove Theorem \ref{C11 regularity of geodesics}.

\begin{proof}[Proof of Theorem \ref{C11 regularity of geodesics}]
To prove Theorem \ref{C11 regularity of geodesics}, it suffices to prove the existence of $C^{1,1}$ solution to the geodesic equation (\ref{Geodesic equation 2}). Our goal is to derive the $C^{1,1}$ estimate for the perturbation geodesic equation (\ref{Perturbation geodesic equation}). First, we define two Hermitian metrics on $N\times[1,\frac{3}{2}]$ by
\begin{equation}\label{C11 regularity of geodesics equation 1}
\omega = 2r^{-2}\omega_{c},~ \ti{\omega} = 2r^{-2}\Omega_{\psi},
\end{equation}
where $\omega_{c}$ is the K\"{a}hler form of $g_{c}$. Clearly, $(N\times[1,\frac{3}{2}],\omega)$ is a compact $(m+1)$-dimensional Hermitian manifold with nonempty smooth boundary. For convenience, let $g$, $\ti{g}$ be the corresponding Riemannian metrics of $\omega$, $\ti{\omega}$. We use $\nabla$ to denote the Levi-Civita connection of $g$. Let $\{e_{i}\}_{i=1}^{m+1}$ be a local frame for $T_{\mathbb{C}}^{(1,0)}M$. Recalling (\ref{Definition of S}), we have
\begin{equation*}
(\de\dbar\psi)(e_{i},\ov{e}_{j}) = \psi_{i\ov{j}}+S_{i\ov{j}}^{p}\psi_{p}+S_{i\ov{j}}^{\ov{p}}\psi_{\ov{p}},~\,
(\de\dbar r)(e_{i},\ov{e}_{j}) =  r_{i\ov{j}}+S_{i\ov{j}}^{p}r_{p}+S_{i\ov{j}}^{\ov{p}}r_{\ov{p}},
\end{equation*}
where $\psi_{i\ov{j}}=(\nabla^{2}\psi)(e_{i},\ov{e}_{j})$ and $r_{i\ov{j}}=(\nabla^{2}r)(e_{i},\ov{e}_{j})$. Since $\frac{\de}{\de r}$ is a vector field on $C(N)$, we write
\begin{equation*}
\frac{\de}{\de r} = V^{p}e_{p}+\ov{V^{p}}\ov{e}_{p}.
\end{equation*}
Combining (\ref{Definition of Omega psi}) and (\ref{C11 regularity of geodesics equation 1}), we compute
\begin{equation}\label{C11 regularity of geodesics equation 2}
\begin{split}
\ti{g}_{i\ov{j}}
= {} & g_{i\ov{j}}+(\de\dbar\psi)(e_{i},\ov{e}_{j})-\frac{\de\psi}{\de r}(\de\dbar r)(e_{i},\ov{e}_{j}) \\
= {} & g_{i\ov{j}}+\psi_{i\ov{j}}+S_{i\ov{j}}^{p}\psi_{p}+S_{i\ov{j}}^{\ov{p}}\psi_{\ov{p}} \\
     & -(V^{p}\psi_{p}+\ov{V^{p}}\psi_{\ov{p}})\left(r_{i\ov{j}}+S_{i\ov{j}}^{p}r_{p}+S_{i\ov{j}}^{\ov{p}}r_{\ov{p}}\right) \\
= {} & g_{i\ov{j}}+\psi_{i\ov{j}}+T_{i\ov{j}}^{p}\psi_{p}+T_{i\ov{j}}^{\ov{p}}\psi_{\ov{p}},
\end{split}
\end{equation}
where
\begin{equation*}
\begin{split}
     T_{i\ov{j}}^{p} = {} & S_{i\ov{j}}^{p}-\left(r_{i\ov{j}}+S_{i\ov{j}}^{p}r_{p}+S_{i\ov{j}}^{\ov{p}}r_{\ov{p}}\right)V^{p}, \\
T_{i\ov{j}}^{\ov{p}} = {} & S_{i\ov{j}}^{\ov{p}}-\left(r_{i\ov{j}}+S_{i\ov{j}}^{p}r_{p}+S_{i\ov{j}}^{\ov{p}}r_{\ov{p}}\right)\ov{V^{p}}.
\end{split}
\end{equation*}
It follows from (\ref{Property of S}) that the tensor field $T$ satisfies (\ref{Property of T}). By (\ref{C11 regularity of geodesics equation 2}), the perturbation geodesic equation (\ref{Perturbation geodesic equation}) is equivalent to
\begin{equation}\label{C11 regularity of geodesics equation 3}
\left\{ \begin{array}{ll}
  \det(g_{i\ov{j}}+\psi_{i\ov{j}}+T_{i\ov{j}}^{p}\psi_{p}+T_{i\ov{j}}^{\ov{p}}\psi_{\ov{p}}) = \ve f\det(g_{i\ov{j}}), \\[2mm]
  g_{i\ov{j}}+\psi_{i\ov{j}}+T_{i\ov{j}}^{p}\psi_{p}+T_{i\ov{j}}^{\ov{p}}\psi_{\ov{p}}>0.
\end{array}\right.
\end{equation}
By \cite[Theorem 1]{GZ12}, there exists a smooth solution of (\ref{Perturbation geodesic equation}), and we denote it by $\psi_{\ve}$. Combining \cite[Theorem 1, Proposition 3]{GZ12} with (\ref{C11 regularity of geodesics equation 1}) and $r\in[1,\frac{3}{2}]$, we obtain
\begin{equation}\label{C11 regularity of geodesics equation 4}
\|\psi_{\ve}\|_{C_{w}^{2}(N\times[1,\frac{3}{2}],g)}+\sup_{\de(N\times[1,\frac{3}{2}])}|\nabla^{2}\psi_{\ve}|_{g} \leq C,
\end{equation}
where $C$ does not depend on $\ve$. By the equivalence of (\ref{Perturbation geodesic equation}) and (\ref{C11 regularity of geodesics equation 3}), $\psi_{\ve}$ is also a smooth solution of (\ref{C11 regularity of geodesics equation 3}). Thanks to Proposition \ref{C11 interior estimate} and (\ref{C11 regularity of geodesics equation 4}), we obtain the $C^{1,1}$ estimate
\begin{equation*}
\sup_{N\times[1,\frac{3}{2}]}|\nabla^{2}\psi_{\ve}|_{g} \leq C,
\end{equation*}
where $C$ does not depend on $\ve$. Letting $\ve\rightarrow0$, we showed existence of $C^{1,1}$ solution to the geodesic equation (\ref{Geodesic equation 2}), as required.
\end{proof}

\end{document}